\documentclass[12pt,draftcls,onecolumn]{IEEEtran}

\usepackage{mathtools} 
\usepackage{latexsym}
\usepackage{graphicx}
\usepackage{array}
\usepackage{amsmath}
\usepackage{amsfonts}
\usepackage{amssymb}
\usepackage{amsthm}
\usepackage{epsfig}
\usepackage{cite}
\usepackage{tabulary}
\usepackage{booktabs}
\usepackage{algorithm,algorithmic}
\usepackage{boxedminipage}
\usepackage{float}
\graphicspath{{figures/}}
\usepackage{graphics}
\usepackage{multirow}
\usepackage{threeparttable} 
\usepackage[usenames,dvipsnames]{color}
\usepackage{url,bm,xspace,dsfont}
\usepackage{verbatim}

\newtheorem{theorem}{Theorem}
\newtheorem{lemma}{Lemma}
\newtheorem{definition}{Definition}
\newtheorem{corollary}{Corollary}

\newtheorem{remark}{Remark}

\newtheorem{assumptions}{Assumptions}
\newtheorem{problem}{Problem}

\newcommand{\eps}{\epsilon}

\newcommand{\veps}{\varepsilon}
\newcommand{\la}{\langle}
\newcommand{\ra}{\rangle}

\newcommand{\re}{\Re}

\newcommand{\sr}{\stackrel}

\newcommand{\rar}{\rightarrow}

\newcommand{\tri}{\sr{\triangle}{=}}

\newcommand{\be}{\begin{equation}}
\newcommand{\ee}{\end{equation}}
\newcommand{\bea}{\begin{eqnarray}}
\newcommand{\eea}{\end{eqnarray}}
\newcommand{\bes}{\begin{eqnarray*}}
\newcommand{\ees}{\end{eqnarray*}}

\newcommand{\bi}{\begin{itemize}}
\newcommand{\ei}{\end{itemize}}
\newcommand{\ben}{\begin{enumerate}}
\newcommand{\een}{\end{enumerate}}

\newcommand{\bp}{\begin{problem}}
\newcommand{\ep}{\end{problem}}
\newcommand{\hso}{\hspace{.1in}}
\newcommand{\hst}{\hspace{.2in}}

\newcommand{\noi}{\noindent}

\newcommand{\bc}{\begin{center}}
\newcommand{\ec}{\end{center}}

\begin{document}


%
\title{Dynamic Team Theory of Stochastic Differential Decision Systems with  Decentralized Noiseless Feedback  Information Structures 
via   Girsanov's Measure Transformation}


\author{ Charalambos D. Charalambous\thanks{C.D. Charalambous is with the Department of Electrical and Computer Engineering, University of Cyprus, Nicosia 1678 (E-mail: chadcha@ucy.ac.cy).} 
}

\maketitle

\begin{abstract}
In this paper we generalized static team theory to dynamic team theory,  in the context of  stochastic differential decision system with decentralized noiseless feedback information structures.  

 We apply Girsanov's theorem to transformed the initial stochastic dynamic team problem to an equivalent team problem, under a reference probability space, with state process and information structures  independent of any of the team decisions. Subsequently,  we show, under certain conditions, that  continuous-time and discrete-time stochastic dynamic team problems,  can be transformed to equivalent  static  team problems, although computing the optimal team strategies using this method might be computational intensive. Therefore, we propose an alternative method, by  deriving team and Person-by-Person (PbP)  optimality  conditions, via the stochastic Pontryagin's maximum  principle,  consisting of forward and backward stochastic differential equations, and a set of conditional variational Hamiltonians with respect to the information structures of  the team members.  
 
 Finally, we relate the backward stochastic differential equation to the value process of the stochastic team problem.

\end{abstract}

  \section{Introduction}
\label{introduction}
In classical stochastic control or decision theory the control actions or decisions applied by the multiple controllers or Decision Makers (DM)  are based on the same information. The underlying assumption is that the acquisition of information is centralized, or the information collected at different observation posts is communicated to each controller or DMs.   
 Classical stochastic control problems are often classified, based on the information available for control actions, into fully observable \cite{benes1971,elliott1977,bismut1978,haussmann1986,bensoussan1992a,ahmed1998,yong-zhou1999} and partially observable \cite{elliott1977,bensoussan1992a,charalambous-hibey1996}. Fully, observable refers to the case when the information structure or pattern available for control actions is generated by the  state process (also called feedback information) or the exogenous state noise process (also called nonanticipative information), and partially observable refers to the case when the information structure available for control actions is a nonlinear function of the state process corrupted by exogenous observation noise process.  

In this paper, we deviate from the classical stochastic control or decision formulation by  consider a system operating over a finite time period $[0, T]$, with  the following features. 

\ben

\item There are $N$  observation posts or stations collecting information; 

\item There are $N$ control stations, each  having direct access to  information collected by at most one   observation post, without delay;

\item  The observation stations may not communicate their information to the other control stations, or they  may communicate their information to  the other control stations by signaling part or all of their information to some of the  control stations with delay;  

\item  The  $N$ control stations may  not   have perfect recall, that is, information which is available at any of the control stations at time $t\in [0, T)$ may not be  available at any future time $\tau \geq t, \tau \in (0,T]$; 

\item  The  control strategies applied at the $N$ control stations have to be coordinated to optimize a common pay-off or reward. 
\een
In the above formulation we have assumed that one observation post is serving one control station without delay, and we allowed the possibility that a subset of the other observation posts  signal their information to any of the control stations they are not serving  subject to delay. Such signaling among the observation posts and control stations is called information sharing \cite{witsenhausen1971,kurtaran1975,yoshikawa1975,varaiya-walrand1978}.  \\   
The elements of the proposed  system of study are the following. 
\begin{align}
&{\mathbb Z}_N \tri \Big\{1, 2, \ldots, N\Big\}: \: \mbox{Set of observation posts/control stations}; \nonumber \\
& x: [0,T] \times \Omega \longrightarrow {\mathbb R}^{n}: \: \mbox{Unobserved state process;}  \nonumber   \\
& W: [0,T] \times \Omega \longrightarrow {\mathbb R}^n: \: \mbox{State exogenous Brownian Motion (BM) process};  \nonumber  \\
&\mbox{For} \:  i=1, \ldots, N, \nonumber \\
& u^i:[0, T ] \longrightarrow {\mathbb A}^i:   \: \mbox{Control process   with action space ${\mathbb A}^i \subseteq {\mathbb R}^{d_i}$ applied at the $i$th control station};   \nonumber \\
& z^i: [0,T] \times \Omega \longrightarrow {\mathbb R}^{k_i}: \: \mbox{ Distributed observation process  collected at the $i$th observation post}; \nonumber\\
& h^i:[0,T] \times C([0,T], {\mathbb R}^n) \longrightarrow {\mathbb R}^{k_i}: \: \mbox{Information functional generating $z^i$ at the $i$th observation post;} \nonumber \\
& {\mathbb U}^{z^i}[0,T]: \:  \mbox{Admissible strategies generating the control actions at the $i$th control station } \nonumber \\
&\mbox{based on $\{z^i(t): t \in [0, T]\}$}; \nonumber \\
& J: {\mathbb A}^{(N)} \longrightarrow (-\infty, \infty]: \: \mbox{Team pay-off or reward.}    \nonumber 
\end{align}
We call as usual the information available as arguments of the control laws, which generate the control actions applied at the $N$ control stations,  ``information Structure or Pattern". 

Suppose, for now,  there is no signaling of information from the observation posts to any of the control station they are not serving, and let $\{z^i(t): 0\leq t \leq T\}$  denote the observation  available to the $i$th control station to generate the control actions   $\{u_t^i: 0 \leq t \leq T\}$ for $i=1,\ldots, N$. Denote the  corresponding  control strategies by ${\mathbb U}^{z^i}[0,T]$, for $i=1,\ldots, N$. 
Given the control strategies, the  performance of the collective decisions or control actions applied by the control stations, the stochastic differential decision  system is formulated using  dynamic team theory, as follows.   
\begin{align}
&\inf \Big\{ J(u^1, \ldots, u^N): (u^1,  \ldots, u^N) \in \times_{i=1}^N {\mathbb U}^{z^i}[0,T]\Big\}, \label{pf1}\\
&J(u^1, \ldots, u^N) = {\mathbb E}^{{\mathbb P}_\Omega} \Big\{ \int_{0}^T \ell(t,x(t),u_t^1(z^1),\ldots, u_t^N(z^N))dt  +  \varphi(x(T))     \Big\} ,  \label{pf2}
\end{align}
subject to stochastic It\^o differential dynamics and  distributed noiseless observations
\begin{align}
dx(t) =& f(t,x(t),u_t^1(z^1), \ldots, u_t^N(z^N))dt + \sigma(t,x(t))dW(t), \hso x(0)=x_0, \; t \in (0,T],  \label{pf3} \\
z^i(t)=&h^i(t,x), \hso t \in [0,T], \; i=1, \ldots, N,  \label{pf4}
\end{align}
where ${\mathbb E}^{{\mathbb P}_\Omega}$ denotes expectation with respect to an underlying   probability space $\Big(\Omega, {\mathbb F}, {\mathbb P}_\Omega\Big)$.  The  stochastic system (\ref{pf3}) may be  
a compact representation of  many interconnected subsystems with states $\{x^i \in {\mathbb R}^{n_i}: i=1, \ldots, N\}, n \tri \sum_{i=1}^N n_i$, aggregated into a single state representation $x \in {\mathbb R}^n$, where $x^i$ represents the state of the local subsystem,   $\{z^i(t): 0\leq t \leq T\}$ its local distributed observation process collected at the $i$th local observation post, and $\{u_t^i(z^i): 0\leq t \leq T\}$ its local control process applied at the $i$th local control station, such that  for each $t \in [0,T]$, the  control law is $u_t^i(z^i) \equiv \mu_t^i(\{z^i(s): 0 \leq s \leq t\})$,  a nonanticipative measurable function $\mu_t^i(\cdot)$ of the $i$th control station   information structure $\{z^i(t): 0\leq t \leq T\}$, for $i=1, \ldots, N$. \\
 We often call the stochastic differential decision system (\ref{pf1})-(\ref{pf4}) with  decentralized noiseless feedback information structures, $\{z^1(t), z^2(t), \ldots, z^N(t): 0\leq t \leq T\}$, a stochastic dynamic team problem, and a strategy $u^o \tri (u^{1,o}, u^{2,o}, \ldots, u^{N,o}) \in \times _{i=1}^N {\mathbb U}^{z^i}[0,T]$ which achieves the infimum in (\ref{pf1}) a team optimal strategy.\\
 Moreover, we call  $u^o \tri (u^{1,o}, u^{2,o}, \ldots, u^{N,o}) \in \times _{i=1}^N {\mathbb U}^{z^i}[0,T]$  a PbP optimal strategy if  
\begin{align}
 J(u^{1,o}, \ldots, u^{N,o}) 
&\leq  J(u^{1,o}, \ldots, u^{i-1,o},u^i, u^{i+1,o}, \ldots, u^{N,o}),  \forall u^i \in {\mathbb U}^i[0,T], \forall i =1,\ldots, N, \label{pp4a}
 \end{align} 
 and the infimum 
subject to constraints (\ref{pf3}), (\ref{pf4}) is achieved. In team theory terminology  $\{u^1, \ldots, u^N\}$ are called the DMs,  agents or members of the team game. 

In this paper, we investigate the stochastic dynamic team problem (\ref{pf1})-(\ref{pf4}), and its generalization when,  there is information sharing from the  observation posts to any of the control stations, and there is  no perfect recall of  information at the control stations.

Recall that a stochastic team problem is called a ``Static Team Problem" 
 if  the information structures available for decisions are not affected by any of the team decisions. Optimality conditions for static team problems are developed by  Marschak and Radner    \cite{marschak1955,radner1962,marschak-radner1972}, and subsequently generalized in   \cite{krainak-speyer-marcus1982a}.  Clearly, since the information structures $\{z^i(t): 0 \leq t \leq T\}, i=1, \ldots, N$ generated by (\ref{pf4}) are affected by the team decisions via the state process $\{x(t):0\leq t \leq T\}$ generated by (\ref{pf3}), the static team  theory optimality conditions given in   \cite{marschak1955,radner1962,marschak-radner1972,krainak-speyer-marcus1982a} donot apply.
 
 On the other hand,  stochastic optimal control theory with full information is developed under a centralized assumption on the information structures.  Therefore, a natural question is whether any of these techniques developed over the last 60 years for centralized stochastic control problems and dynamic games, such as, dynamic programming, stochastic Pontryagin's maximum principle, and martingale methods are  applicable  to stochastic dynamic team problems, and if so how. 

In this paper we apply  techniques from classical strochastic control theory to  generalize Marschak's and Radner's static team theory   \cite{marschak1955,radner1962,marschak-radner1972} to continuous-time stochastic differential decision systems with decentralized  noiseless feedback  information structures, defined by (\ref{pf1})-(\ref{pf4}). Moreover, we discuss generalizations of (\ref{pf1})-(\ref{pf4}), when there is information sharing from the observation posts to any of the control stations, and there is no perfect recall of information at the control stations. Our methodology is based on deriving team and PbP optimality conditions,  using stochastic Pontryagin's maximum principle, by  utilizing the semi martingale representation method due to Bismut \cite{bismut1978}, under a weak formulation of the probability space by invoking   Girsanov's theorem \cite{liptser-shiryayev1977}.   First, we  apply  Girsanov's theorem to   transform the original stochastic  dynamic team problem to  an equivalent   team problem,  under  a reference probability space in which the state process and the information structures are  not affected by any the team decisions. Subsequently,  we show the precise connection between Girsanov's measure transformation and Witsenhausen's notion of  ``Common Denominator Condition" and ``Change of Variables" introduced in \cite{witsenhausen1988} to establish equivalence between static and dynamic team problems. We elaborate on  this connection for  both continuous-time and discrete-time stochastic systems, and we state certain results from static team theory which are directly applicable. However, since the computation of the optimal team strategies via static team theory might be computationally intensive, we proceed further to derive optimality conditions based on stochastic variational methods, by taking advantage of the fact that under the reference measure,  the state process and the information structures do not react to any perturbations of the team decisions. The  optimality conditions are given by a ``Hamiltonian System" consisting of a backward and forward stochastic differential equations, while the optimal team actions of the $i$th team member  are  determined by  a conditional variational Hamiltonian, conditioned on the information structure of the $i$th team member, while the rest are  fixed to their optimal values, for $i=1, \ldots, N$. Finally, we show  the connection between  the backward stochastic differential equation and the value process of the stochastic dynamic team problem. 

We  point out that the  approach we pursued in this paper is different from the various approaches pursued over the years to address stochastic dynamic decentralized decision systems, formulated using team theory  in 
\cite{ho-chu1972,ho-chu1973,kurtaran-sivan1973,sandell-athans1974,yoshikawa1975,kurtaran1975,varaiya-walrand1977,varaiya-walrand1978,bagghi-basar1980,ho1980,krainak-speyer-marcus1982a,krainak-speyer-marcus1982b,walrand-varaiya1983,walrand-varaiya1983a,basar1985,bansal-basar1987,aicardi-davoli-minciardi1987,witsenhausen1988,waal-vanschuppen2000,bamieh-voulgaris2005,teneketzis2006,mahajan-teneketzis2009,mahajan-teneketzis2009a,nayyar-mahajan-teneketzis2011,nayyar-teneketzis2011,nayyar-teneketzis2011a,vanschuppen2011,lessard-lall2011,vanschuppen2012,farokhi-johansson2012,mishra-langbort-dullerud2012,gattami-bernhardsson-rantzer2012}, and our recent treatment in \cite{charalambous-ahmedFIS_Parti2012}. Compared to  \cite{charalambous-ahmedFIS_Parti2012},  in the current paper we apply Girsanov's measure transformation, which allows us to derive the stochastic Pontryagin's maximum principle, for decentralized noiseless feedback information structures, instead of  nonanticipative (open loop) information structures  adapted to a sub-filtration of the fixed filtration  generated by  the Brownian motion $\{W(t): t \in [0,T]\}$ (e.g., $u_t=\mu(t,W)$) considered in  \cite{charalambous-ahmedFIS_Parti2012}. Since feedback strategies are more desirable  compared to nonanticipative strategies, this paper is an improvement to  \cite{charalambous-ahmedFIS_Parti2012}. The only disadvantage is that, unlike    \cite{charalambous-ahmedFIS_Parti2012}, we cannot allow dependence of the diffusion coefficient $\sigma$ on the team decisions. The current paper also generalizes some of the results on centralized partial information   optimality conditions derived in \cite{ahmed-charalambous2012a} to centralized partial feedback  information.    
We note that the case of decentralized noisy information structures is treated in \cite{charalambous-ahmedPISG2013a}, and therefore by combining the results of this paper with those in \cite{charalambous-ahmedPISG2013a}, we can handle any combination of decentralized noiseless and noisy information structures. However, when applying the optimality conditions to determine the optimal team strategies, the main challenge  is the computation of conditional expectations with respect to  the information structures.   The procedure is similar to \cite{charalambous-ahmedFIS_Partii2012}, where various examples  from the communication and control areas are presented,  using decentralized  nonanticipative strategies.

The rest of the paper is organized as follows. In Section~\ref{formulation} we introduce   the stochastic differential team problem and its equivalent re-formulations using the weak Girsanov measure transformation approach. Here, we also establish the connection between Girsanov's theorem and Witsenhausen's ``Common Denominator Condition" and ``Change of Variables" \cite{witsenhausen1988} for stochastic continuous-time and discrete-time dynamical systems.  Further,  we derive the variational equation which we invoke in Section~\ref{optimality}, to   derive the optimality conditions, both under the reference probability measure and under the initial probability measure.  In Section~\ref{cf} we provide  concluding remarks and  comments on future work.\\

 \section{Equivalent Stochastic Dynamic Team Problems}
 \label{formulation}

In this section, we consider the stochastic  dynamic team problem (\ref{pf1})-(\ref{pf4}), and we apply   Girsanov's theorem, to transformed it to an equivalent team problem under  a reference probability measure, in which the information structures are functionals of Brownian motion, and hence  independent of any of the team decisions. 
We will also briefly discuss the  discrete-time counterpart of Girsanov's theorem, and we will show  its equivalence  to  Witsenhausen's so-called ``Common Denominator Condition" and ``Change of Variables" discussed in \cite{witsenhausen1988}. 

 Let $\Big(\Omega,{\mathbb F},  \{ {\mathbb F}_{0,t}:   t \in [0, T]\}, {\mathbb P}\Big)$ denote a complete filtered probability space satisfying the usual conditions, that is,  $(\Omega,{\mathbb F}, {\mathbb P})$ is complete, ${\mathbb F}_{0,0}$ contains all ${\mathbb P}$-null sets in ${\mathbb F}$. Note that  filtrations $\{{\mathbb F}_{0,t} : t \in [0, T]\}$ are mononote in the sense that ${\mathbb F}_{0,s} \subseteq {\mathbb F}_{0,t}$, $\forall 0\leq s \leq t \leq T$. Moreover,  we assume throghout that filtrations $\{ {\mathbb F}_{0,t}: t \in [0, T] \}$ are  right continuous, i.e.,  ${\mathbb F}_{0,t} = {\mathbb F}_{0,t+} \tri \bigcap_{s>t} {\mathbb F}_{0,s}, \forall t \in [0,T)$.  We define ${\mathbb F}_T \tri \{{\mathbb F}_{0,t}: t \in [0, T]\}$. 
  \noi Consider a random process $\{z(t):  t \in [0,T]\}$ taking values in $({\mathbb Z}, {\cal B}({\mathbb Z}))$, where $({\mathbb Z}, d)$ is a metric space, defined   on  the filtered probability space $(\Omega,{\mathbb F},\{{\mathbb F}_{0,t}: t \in [0,T]\}, {\mathbb P})$.  
  The process $\{z(t): t \in [0,T]\}$ is said to be (a) measurable,  if the map $(t,\omega)  \rar  z(t,\omega)$ is  ${\cal B }([0,T]) \times{\mathbb  F}/{\cal B}({\mathbb Z})-$measurable, (b) $\{{\mathbb F}_{0,t}: t \in [0,T]\}-$adapted, if for all $t \in [0,T]$, the map $\omega \rar z(t,\omega)$ is  ${\mathbb F}_{0,t}/{\cal B}({\mathbb Z})-$measurable, (c) $\{{\mathbb F}_{0,t}: t \in [0,T]\}-$progressively measurable if for all $t \in [0,T]$, the map $(s,\omega) \rar z(s,\omega)$ is  $ {\cal B}([0,t]) \otimes {\mathbb F}_{0,t}/{\cal B}({\mathbb Z})-$measurable.  It can be shown that any stochastic process  $\{z(t):  t \in [0,T]\}$  on a filtered probability space  $(\Omega,{\mathbb F},  \{ {\mathbb F}_{0,t}:   t \in [0, T]\}, {\mathbb P})$ which is measurable and adapted has a progressively measurable modification. Unless otherwise specified, we shall say a process $\{z(t):  t \in [0,T]\}$ is $\{{\mathbb F}_{0,t}: t \in [0,T]\}-$adapted if the processes is $\{{\mathbb F}_{0,t}: t \in [0,T]\}-$progressively measurable \cite{yong-zhou1999}.

We use the following notation.

\begin{table}[htbp]\caption{Table of Notation}
\centering
\begin{tabular}{r c p{8cm} }
\toprule
& & ${\mathbb Z}_N  \tri \{1,2,\ldots, N\}$:  Subset of natural numbers.\\
& & $s\tri \{s^1, s^2, \ldots,\ldots, s^N\}$: Set consisting of $N$ elements. \\
& & $s^{-i}=s \setminus\{s^i\}, \hso s=(s^{-i},s^i)$: Set $s$ minus $\{s^i\}$.  \\
& & ${\cal L}({\cal X},{\cal Y})$: Linear transformation mapping a vector space ${\cal X}$ into a vector space ${\cal Y}$. \\
& & $ A^{(i)}$: $i$th column of a map $A \in  {\cal L}({\mathbb R}^n,{\mathbb R }^m), i=1, \ldots, n$. \\
& & $ {\mathbb A}^i \subseteq {\mathbb R}^{d_i}$: Action spaces of controls applied at the $i$th control station, $i=1, \ldots, N$.\\
\bottomrule
\end{tabular}
\label{notations}
\end{table}

Let $C([0,T], {\mathbb R}^n)$ denote the space of continuous real-valued $n-$dimensional  functions defined on the time interval $[0,T]$, and ${\cal B}({\mathbb R}^n)$ its canonical Borel filtration. 
   
  Let $L_{{\mathbb F}_T}^2([0,T],{\mathbb R}^n) \subset   L^2( \Omega \times [0,T], d{\mathbb P}\times dt,  {\mathbb R}^n) \equiv L^2([0,T], L^2(\Omega, {\mathbb R}^n)) $ denote the space of ${\mathbb F}_T-$adapted random processes $\{z(t): t \in [0,T]\}$   such that
\bes
{\mathbb  E}\int_{[0,T]} |z(t)|_{{\mathbb R}^n}^2 dt < \infty,
\ees   
 which is a Hilbert subspace of     $L^2([0,T], L^2(\Omega, {\mathbb R}^n))$.      \\
  Similarly, let $L_{{\mathbb F}_T}^2([0,T],  {\cal L}({\mathbb R}^m,{\mathbb R}^n)) \subset L^2([0,T] , L^2(\Omega, {\cal L}({\mathbb R}^m,{\mathbb R}^n)))$ denote the space of  ${\mathbb  F}_{T}-$adapted $n\times m$ matrix valued random processes $\{ \Sigma(t): t \in [0,T]\}$ such that
 \bes
  {\mathbb  E}\int_{[0,T]} |\Sigma(t)|_{{\cal L}({\mathbb R}^m,{\mathbb R}^n)}^2 dt  \tri  {\mathbb E} \int_{[0,T]} tr(\Sigma^*(t)\Sigma(t)) dt < \infty.
  \ees

Let $B_{{\mathbb F}_T}^{\infty}([0,T], L^2(\Omega,{\mathbb R}^{n}))$ denote the space of ${\mathbb F}_{T}$-adapted ${\mathbb R}^{n}-$ valued second order random processes endowed with the norm topology  $\parallel  \cdot \parallel$ defined by  
\bes
 \parallel \phi \parallel^2  \tri \sup_{t \in [0,T]}  {\mathbb E}|\phi(t)|_{{\mathbb R}^{n}}^2.
 \ees

Next, we introduce conditions on the coefficients $\{f, \sigma, h^i, i=1, \ldots, N\}$, which are partly used to derive the results of this section.

\begin{assumptions}(Main assumptions)
\label{A1-A4}
The  drift $f$, diffusion coefficients $\sigma$, and information functional $h^i$  are     Borel measurable  maps: 
\begin{align}
 f: [0,T] \times {\mathbb R}^n & \times {\mathbb A}^{(N)} \longrightarrow {\mathbb R}^n , \;  \sigma: [0,T] \times {\mathbb R}^n  \longrightarrow {\cal L}({\mathbb R}^n, {\mathbb R}^n), \nonumber \\
 &h^i : [0,T] \times C([0,T],  {\mathbb R}^n) \longrightarrow {\mathbb R}^{k_i}, \hso \forall i \in {\mathbb Z}_N.\nonumber 
 \end{align}
Moreover, 

{\bf (A0)} ${\mathbb A}^i \subseteq {\mathbb R}^{d_i}$ is nonempty, $\forall i \in {\mathbb Z}_N$.

\noi There exists a $K >0$ such that


{\bf (A1)} $|f(t,x,u)|_{{\mathbb R}^n} \leq K (1 + |x|_{{\mathbb R}^n} +|u|_{{\mathbb R}^d}),   \forall t \in [0,T]$;

{\bf (A2)} $|\sigma(t,x)|_{{\cal L}({\mathbb R}^n, {\mathbb R}^n)} \leq K (1+ |x|_{{\mathbb R}^n}  ),    \forall t \in [0,T]$;

{\bf (A3)} $|\sigma(t,x)-\sigma(t,y)|_{{\cal L}({\mathbb R}^n, {\mathbb R}^n)} \leq K |x-y|_{{\mathbb R}^n},  \forall t \in [0,T]$;

{\bf (A4)} $\sigma(t,x)$ is invertible  $\forall (t,x) \in [0,T]\times {\mathbb R}^n$;

{\bf (A5)} $|\sigma^{-1}(t,x)f(t,x,u)|_{{\mathbb R}^n}^2 < K,$ uniformly in $(t, x,u) \in [0, T] \times {\mathbb R}^n \times {\mathbb A}^{(N)}$;

{\bf (A6)} $|\sigma(t,x)|_{ {\cal L}({\mathbb R}^n, {\mathbb R}^n)} \geq K (1 + |x|_{{\mathbb R}^n}^q),  \forall t \in [0, T], \hso q\geq 1$; 

{\bf (A7)} $|\sigma^{-1}(t,x) f(t,x,u)|_{{\mathbb R}^n} \leq K (1 + |x|_{{\mathbb R}^n} +|u|_{{\mathbb R}^d}),   \forall t \in [0,T]$;

{\bf (A8)} $|\sigma^{-1}(t,x)f(t,x,u)- \sigma^{-1}(t,z)f(t,z,v)|_{{\mathbb R}^n} \leq  (1 + |x-z|_{{\mathbb R}^n}  + |u-v|_{{\mathbb R}^d} )$.


\end{assumptions}

\subsection{Equivalent Stochastic Team Problems via Girsanov's }
\label{full}
Next,  we define the dynamic team problem (\ref{pf1})-(\ref{pf4}) using the weak  Girsanov's change of measure approach.

  We start with a canonical space $\Big(\Omega, {\mathbb F}, {\mathbb P}\Big)$ on which   $(x_0, \{W(t): t \in [0,T]\})$  are defined by 

{\bf (WP1)} $x(0)=x_0$:  an ${\mathbb R}^n$-valued  Random Variable with distribution $\Pi_0(dx)$;

{\bf (WP2)} $\{ W(t): t \in [0,T]\}$: an  ${\mathbb R}^{m}$-valued  standard Brownian motion,  independent of $x(0)$;

We introduce the Borel $\sigma-$algebra $  {\cal B}(C[0,T], {\mathbb R}^n))$  on $C([0,T], {\mathbb R}^n)$ generated by $\{W(t): 0\leq t \leq T\}$, and let   ${\mathbb P}^W$ its  Wiener mesure on it. Further, we introduce the filtration ${\cal F}_T^W\tri \{  {\cal F}_{0,t}^W: t \in [0,T]\}$   generated by truncations of $W \in C([0,T], {\mathbb R}^n)$. That is, ${\cal F}_{0,t}^W$ is the sub-$\sigma$-algebra generated by the family of sets
\begin{align}
\Big\{ \{ W \in C([0,T], {\mathbb R}^n): W(s) \in A \}:   \; 0\leq s \leq t, \; A \in {\cal B}({\mathbb R}^n) \Big\}, \hst t \in [0,T],
\end{align}
which implies  $\{{\cal F}_{0,t}^{W}: t \in [0,T]\}$ is the canonical Borel filtration,  ${\cal F}_{0,T}^{W}= {\cal B}(C[0,T], {\mathbb R}^n))$, and ${\cal F}_{0,t}^{W}= {\cal B}_t(C[0,T], {\mathbb R}^n))$ are the truncations for $t \in [0,T]$.
Next, we   define
\begin{align} 
\Omega \tri & {\mathbb R}^n \times C([0,T], \re^n), \hso
{\mathbb F} \tri  {\cal B}({\mathbb R}^n) \otimes {\cal B}( C([0,T], {\mathbb R}^n)), \nonumber \\
 {\mathbb F}_{0,t} \tri &   {\cal B}({\mathbb R}^n) \otimes  {\cal F}_{0,t}^W,  \hso 
{\mathbb P} \tri  \Pi_0 \times {\mathbb P}^W. \nonumber 
\end{align}
On the probability space $(\Omega, {\mathbb F}, \{{\mathbb F}_{0,t}: t \in [0, T]\}, {\mathbb P})$ we 
define the stochastic differential equation
\bea
dx(t)=\sigma(t,x(t))dW(t), \hst x(0)=x_0. \hst t \in (0,T]. \label{fi1}
\eea
Then by Assumptions~\ref{A1-A4}, {\bf (A2), (A3)}, and for any initial condition satisfying  ${\mathbb E} |x(0)|_{{\mathbb R}^n}^q, q \geq 1$,  (\ref{fi1}) has a unique strong solution \cite{liptser-shiryayev1977},  $x(\cdot) \in C([0,T], {\mathbb R}^n)-{\mathbb P}-a.s.$ adapted to the filtration $\{{\mathbb F}_{0,t}: t \in [0, T]\}$, and $x(\cdot) \in B_{{\mathbb F}_T}^\infty([0,T], L^2(\Omega, {\mathbb R}^n))$. \\
We also introduce the $\sigma-$algebra  ${\cal F}_{0,t}^x$  defined by
\begin{align}
{\cal F}_{0,t}^x \tri \sigma\Big\{ \{ x \in C([0,T], {\mathbb R}^n):  
x(s) \in A \}:  0\leq s \leq t, \hso A \in {\cal B}({\mathbb R}^n) \Big\}\equiv {\cal B}_t(C([0,T], {\mathbb R}^n)). \label{borel1}
\end{align}
Hence, ${\cal F}_T^x\tri \{{\cal F}_{0,t}^{x}: t \in [0,T]\}$ is the canonical Borel filtration generated by $x(\cdot) \in C([0,T], {\mathbb R}^n)-{\mathbb P}-a.s.$ satisfying (\ref{fi1}).
\noi From  (\ref{fi1}), and the additional Assumptions~\ref{A1-A4}, {\bf  (A4)}  on $\sigma$, it can be shown  that ${\cal F}_{0,t}^x ={\mathbb F}_{0,t} \equiv {\cal F}^{x(0)} \bigvee {\cal F}_{0,t}^{W}, \forall t \in [0,T]$, and this $\sigma-$algebra is independent of any of the team decisions $u$. Unlike  \cite{charalambous-ahmedFIS_Parti2012}, where we utilize open loop or nonanticipative information structures, here we use feedback information structures. Note that for the feedback information structures to be independent of any of the team decisions $u$, it is necessary that under the reference probability measure, ${\mathbb P}$ the state process $x(\cdot)$ is independent of $u$, which is indeed the case because we have restricted  the class of   diffusion coefficients $\sigma$ to those which are  independent of $u$.

Next we prepare to define three  sets of admissible team strategies.  We define the Borel $\sigma-$algebras generated by projections of $x \in {\mathbb R}^n$ on any of its subspaces say, $x^i \tri {\bf \Pi}^i(x)$, and  the distributed observations process $\{z^i(t) \tri h^i(t,x): t \in [0, T]\}$ as follows. 
\begin{align}
{\cal G}^{x^i(t)} \tri & \sigma\Big\{ \{ x^i \in C([0,T], {\mathbb R}^{n_i}): x^i(t) \in A \}:  
 A \in {\cal B}({\mathbb R}^{n_i}) \Big\}, \hso t \in  [0, T], \hso \forall i \in {\mathbb Z}_N \label{npr} \\
 {\cal G}_{0,t}^{z^i}\tri &z^{i,-1} \Big( {\cal B}_t([0,T]) \otimes {\cal B}_t(C([0,T], {\mathbb R}^n))\Big),\; 
  t \in [0,T], \; \forall i \in {\mathbb Z}_N. \label{borel2}
\end{align}
Further, define  ${\cal G}_{T}^{z^i} \tri \{  {\cal G}_{0,t}^{z^i}: t \in [0, T]\}, {\cal G}_{0,t}^{z^i} \subseteq {\cal F}_{0,t}^x, \forall t \in [0,T]\}$,  the canonical Borel  filtration generated by   $\{ z^i(t): 0\leq t \leq T\}$, for $i=1, \ldots, N$.
Define the delayed  sharing information structure at the $i$th control station by ${\cal G}_{T}^{I^i} \tri \{  {\cal G}_{0,t}^{I^i}: t \in [0, T]\}$, which is  the minimum filtration   generated by the Borel $\sigma-$algebra at the $i$th observation post $\{z^i(s): 0\leq s \leq t \}$, and the delayed sharing information signaling, $\Big\{z^j(s-\eps_j): \eps_j >0,  j  \in    {\cal O}(i), 0\leq s \leq  t \Big\}$, $t \in [0,T]$, from the observation posts ${\cal O}(i) \subset \{1,2, \ldots, i-1, i+1, \ldots, N\}$, to the control station $i$,  for     $i=1, \ldots, N$.  

Next, we define the three classes of information structures we consider in this paper.

\begin{definition}(Noiseless Feedback Admissible Strategies)\\
\label{strategiesf}   
{\bf Without Signaling:} If there is no signaling from the observation posts to any of the other control stations, the set of admissible  strategies at the $i$th control station is defined by  
\begin{align}
 {\mathbb U}^{z^i} &[0, T] \tri   \Big\{   u^i : [0,T] \times \Omega \longrightarrow  {\mathbb A}^i \subseteq {\mathbb R}^{d_i}:  \nonumber \\
 &   u^i  \: \mbox{is $\{{\cal G}_{0,t}^{z^i}: t \in [0, T]\}-$Progressively Measurable (PM)} \nonumber \\ &\mbox{and} \hst    {\mathbb E}\int_{0}^T \Lambda^u(t) |u_t|_{{\mathbb R}^d}^2dt < \infty \Big\}, \hso \forall i \in {\mathbb Z}_N.
 \label{cs1aaa}
     \end{align}
A team strategy is    an $N$ tuple  defined by $(u^1,u^2, \ldots, u^N) \in {\mathbb U}^{(N), z}[0,T] \tri \times_{i=1}^N {\mathbb U}^{z^i}[0,T].$\\
     
{\bf With  Signaling:}  If there is delayed sharing information signaling from the other observation posts,
     the set of admissible strategies  at the $i$th control station  is  defined by
\begin{align}
 {\mathbb U}^{I^i}&[0, T] \tri  \Big\{   u^i : [0,T] \times \Omega \longrightarrow  {\mathbb A}^i \subseteq {\mathbb R}^{d_i}:  \nonumber \\
 & u^i  \: \mbox{is $\{{\cal G}_{0,t}^{I^i}: t \in [0, T]\}-$PM and} \hst    {\mathbb E}\int_{0}^T \Lambda^u(t) |u_t|_{{\mathbb R}^d}^2dt < \infty \Big\}, \;   \forall i \in {\mathbb Z}_N. \label{cs1a}
     \end{align}
 A team strategy is an $N$  defined by  $(u^1,u^2, \ldots, u^N) \in {\mathbb U}^{(N)}[0,T] \tri \times_{i=1}^N {\mathbb U}^{I^i}[0,T].$\\
 
{\bf Without Perfect Recall $\sim$ Markov:} If the  distributed observation process collected  at the  $i$th observation post is $z^i=x^i$, and there is no perfect recall, the 
set of admissible strategies  at the $i$th control station  is  defined by
\begin{align}
 {\mathbb U}^{x^i}&[0, T] \tri \Big\{   u^i : [0,T] \times {\mathbb R}^{n_i} \longrightarrow  {\mathbb A}^i \subseteq {\mathbb R}^{d_i}:  \nonumber \\
 & \mbox{for any $t \in [0,T],$} \hso  u_t^i  \: \: \mbox{is ${\cal G}^{x^i(t)}-$measurable and} \hst    {\mathbb E}\int_{0}^T \Lambda^u(t) |u_t|_{{\mathbb R}^d}^2dt < \infty \Big\}, \;   \forall i \in {\mathbb Z}_N. \label{cs1ag}
     \end{align}
 A team strategy is an $N$  defined by $(u^1,u^2, \ldots, u^N) \in {\mathbb U}^{(N), x}[0,T] \tri \times_{i=1}^N {\mathbb U}^{x^i}[0,T].$     
\end{definition} 
 
The results derived in this paper hold for other variations of the  information structures, such as, control stations without perfect recall based on delayed information ${\cal G}^{x^i(t-\delta_i)},\delta_i \geq 0,  i=1, \ldots, N$, etc.  
 
 The reason for imposing the condition  ${\mathbb E}\int_{0}^T \Lambda^u(t) |u_t|_{{\mathbb R}^d}^2dt < \infty$ will be clarified shortly.  
Thus, an admissible strategy, say,  $u\equiv (u^1, \ldots, u^N) \in {\mathbb U}^{(N)}[0, T]$ is a family of $N$ functions, say, $\Big(\mu_t^1(\cdot), \mu_t^2(\cdot), \ldots, \mu_t^N(\cdot)\Big), t \in [0,T]$,  which are progressively measurable (nonanticipative)  with respect to the delayed sharing noiseless feedback  information structure $\{  {\cal G}_{0,t}^{I^i}: t \in [0, T]\} , i=1,2, \ldots, N$. 

Next, for any  $u \in {\mathbb U}^{(N)}[0,T]$ (we can also consider ${\mathbb U}^{(N),z}[0,T] ,  {\mathbb U}^{(N),x}[0,T]$)     we define on $\Big( \Omega,  {\mathbb F}, \{ {\mathbb F}_{0,t}: t \in [0, T]\},  {\mathbb P}\Big)$ the exponential function 
\begin{align}
&\Lambda^{u}(t) \tri  \exp \Big\{ \int_{0}^t f^{*}(s,x(s),u_s) a^{-1}(s,x(s)) dx(s)  \nonumber \\
&-\frac{1}{2}  \int_{0}^t f^{*}(s,x(s)) a^{-1}(s,x(s)) f(s,x(s))ds \Big\}, \hso a(t,x) =\sigma(t,x) \sigma^*(t,x), \hso t \in [0,T]. \label{fi4}  
\end{align}
Under the additional Assumptions~\ref{A1-A4}, {\bf (A5)},    by It\^o's differential rule   $\{\Lambda^u(t): t \in [0, T]\}$   it is the unique $\{{\mathbb F}_{0,t}: t \in [0,T]\}-$adapted, ${\mathbb P}-$a.s. continuous solution \cite{liptser-shiryayev1977} of the stochastic differential equation  
\bea
d \Lambda^{u}(t) = \Lambda^{u}(t) f^{*}(t,x(t),u_t)a^{-1}(t,x(t)) dx(t), \; \Lambda^{u}(0)=1. \label{fi5}
\eea
Given any $u \in {\mathbb U}^{(N)}[0,T]$ we define the reward  of the team game under $\Big(\Omega, {\mathbb  F}, \{{\mathbb  F}_{0,t}: t \in [0, T]\}, {\mathbb P}\Big)$ by 
 \begin{align} 
  J(u^*) \tri \inf_{u \in {\mathbb U}^{(N)}[0,T]}
    {\mathbb E} \biggl\{   \int_{0}^{T}   \Lambda^u(t) \ell(t,x(t),u_t) dt  +  \Lambda^u(T) \varphi(x(T)\biggr\}, \label{fi6}
  \end{align} 
  where $\ell: [0, T] \times {\mathbb R}^n \times {\mathbb A}^{(N)} \longrightarrow (-\infty, \infty] , \varphi : [0,T] \times {\mathbb  R}^n \longrightarrow (-\infty, \infty]$ will be such that (\ref{fi6}) is finite.\\
  For any admissible strategy $u \in {\mathbb U}^{(N)}[0,T]$, by Assumptions~\ref{A1-A4}, {\bf (A5)}, Novikov condition \cite{liptser-shiryayev1977}     
\bea
{\mathbb E} \exp  \Big\{\frac{1}{2} \int_{0}^t |\sigma^{-1}(s, x(s)) f(s,x(s),u_s)|_{{\mathbb R}^n}^2 ds\Big\} < \infty, \label{novikov}
\eea
  which is sufficient for    $\{ \Lambda^{u}(t): 0 \leq t \leq T\}$ defined by (\ref{fi4}) to be  an $\Big( \{{\mathbb F}_{0,t}: t \in [0,T]\}, {\mathbb P}\Big)$-martingale,  $\forall t \in [0, T]$. Thus, by the martingale property, $\Lambda^u(\cdot)$  has constant expectation, 
 $\int_{\Omega} \Lambda^u(t,\omega) d{\mathbb P}(\omega)=1, \forall t \in [0,T]$, and  therefore, we can utilize $\Lambda^u(\cdot)$ which represents a version of  the Radon-Nikodym derivative, to define  a probability measure ${\mathbb P}^u$ on $\Big(\Omega, {\mathbb F}, \{{\mathbb F}_{0,t}: t \in [0,T]\}\Big)$ by setting 
\bea
\frac{ d{\mathbb P}^u}{ d {{\mathbb P}}} \Big{|}_{ {\mathbb F}_{0,t}} = \Lambda^u(t), \hst  t \in [0, T].    \label{fi7}
\eea     
Moreover, by Girsanov's theorem under the probability space $\Big(\Omega, {\mathbb  F}, \{ {\mathbb F}_{0,t}: t \in [0, T]\},  {\mathbb P}^u\Big)$, the process $\{W^u(t): t\in [0,T]\}$ is a standard Brownian motion and it is defined by
\bea
W^u(t) \tri W(t)-\int_{0}^t \sigma^{-1}(s,x(s))f(s,x(s),u_s)ds,  \label{fi8}
\eea
$ t \in [(0,T]$,
and the distribution of $x(0)$ is unchanged. \\
Therefore, under Assumptions~\ref{A1-A4}, {\bf (A1)-(A5)}   we have constructed  the probability space  $\Big(\Omega,{\mathbb F},\{ {\mathbb F}_{0,t}:   t \in [0, T]\}, {\mathbb P}^u\Big)$,  the  Brownian motion $\{W^u(t): t \in [0,T]\}$ defined on it, and the state process $x(\cdot)$ which  is a weak solution of 
\bea
dx(t)=f(t,x(t),u_t)dt + \sigma(t,x(t))dW^u(t), \hso x(0)=x_0, \label{fi9}
\eea
$t \in (0,T]$, unique in probability law defined via (\ref{fi7}),  having the properies $x(\cdot) \in C([0,T], {\mathbb R}^n)-{\mathbb P}^u-a.s.$, it is $\{{\mathbb F}_{0,t}: t \in [0, T]\}-$adapted,  and $x(\cdot) \in B_{{\mathbb F}_T}^\infty([0,T], L^2(\Omega, {\mathbb R}^n))$.  \\
By substituting (\ref{fi7}) into (\ref{fi6}), under the probability measure ${\mathbb P}^u$,  the team game reward is given by
\bea
J(u^*) =\inf_{u \in {\mathbb U}^{(N)}[0,T]}
     {\mathbb E}^u \Big\{ \int_{0}^T \ell(t,x(t),u_t)dt + \varphi(x(T))\Big\} . \label{fi10}
\eea
From the definition of  the Radon-Nikodym derivative (\ref{fi7}), for any admissible strategy, say, $u\in {\mathbb U}^{(N)}[0,T]$ we also have  ${\mathbb E} \int_{0}^T \Lambda^u(t) |u_t|_{{\mathbb R}^d}^2 dt ={\mathbb E}^u \int_{0}^T  |u_t|_{{\mathbb R}^d}^2 dt <\infty$.

Note that if we start with the stochastic dynamic team problem (\ref{fi9}), (\ref{fi10}), the reverse change of measure is obtained as follows. Define 
\begin{align}
 \rho^{u}(T)\tri \Lambda^{u,-1}(t) 
   = &\exp \Big\{- \int_{0}^t f^{*}(s,x(s),u_s) \sigma^{-1}(s,x(s)) dW^u(s) \nonumber \\
   & -\frac{1}{2}  \int_{0}^t |\Big(\sigma(s,x(s))\Big)^{-1}f(s,x(s),u_s)|_{{\mathbb R}^n}ds \Big\}.
\label{gs11}
\end{align}
Then ${\mathbb E} \Lambda^u(t) = {\mathbb E}^u \rho^u(t) \Lambda^u(t)=1, \hst \forall t \in [0,T],$
and  $\frac{ d{\mathbb P}}{ d {{\mathbb P}^u}} \Big{|}_{ {\mathbb F}_{0,t}} = \rho^{ u}(t), \hst t \in [0,T].$
Consequently, under the reference measure $\Big(\Omega, {\mathbb  F}, \{ {\mathbb F}_{0,t}: t \in [0, T]\},  {\mathbb P}\Big)$ the stochastic dynamic team problem is (\ref{fi1}), (\ref{fi5}), (\ref{fi6}).

\begin{remark}
\label{rem-equiv}
We have shown that  under the Assumptions~\ref{A1-A4}, {\bf  (A1)-(A5)}, and $E| x(0)|_{{\mathbb R}^n}< \infty$, for  any $ u \in {\mathbb U}^{(N)}[0,T]$ then ${\mathbb E}\Big(\Lambda^u(t)\Big)=1, \forall t \in  [0,T]$, and that we have two equivalent  formulations of the stochastic dynamic team problem.

{\bf (1)}  Under the original probability space $\Big(\Omega,{\mathbb F},\{ {\mathbb F}_{0,t}:   t \in [0, T]\}, {\mathbb P}^u\Big)$ the dynamic team problem  is  described by the  pay-off (\ref{fi10}), and the $\{{\mathbb F}_{0,t}: t \in [0,T]\}-$adapted continuous strong  solution  $x(\cdot)$ satisfying  (\ref{fi9}), where  the distributed observations collected at the observation posts  $\{z^i(t)=h^i(t,x): t \in [0,T]\}$ are affected by the team decisions via $\{x(t): t \in [0, T]\}$.

{\bf (2)}  Under the reference probability space $\Big(\Omega,{\mathbb F},\{ {\mathbb F}_{0,t}:   t \in [0, T]\}, {\mathbb P}\Big)$ the dynamic team problem  is described by the pay-off (\ref{fi6}), and the $\{{\mathbb F }_{0,t}: t \in [0, T]\}-$adapted continuous pathwise solution of $(x(\cdot), \Lambda(\cdot))$,  satisfying  (\ref{fi1}), (\ref{fi5}), where $\{z^i(t)=h^i(t,x): t \in [0,T]\}$ is not affected by any of the team  decisions. Note that strong uniqueness holds for solutions of (\ref{fi1}), (\ref{fi5}) because both satisfy the Lipschitz conditions (i.e. Assumptions~\ref{A1-A4}, {\bf  (A3), (A5)} hold).
\end{remark}

\begin{remark}
\label{rem-cond}
The Assumptions~\ref{A1-A4}, {\bf (A5)} is satisfied if the following alternative conditions hold. 

{\bf (A5)(a)}   {\bf (A4), (A6) } holds and either (i) {\bf (A1)} is replaced by  $|f(t,x,u)|_{{\mathbb R}^n} \leq K (1 + |x|_{{\mathbb R}^n}),  K>0,  \forall t \in [0,T],$
or (ii) ${\mathbb A}^{(N)}$ is bounded;


\end{remark}

\begin{remark}
\label{wits}
 The Girsanov's measure transformation  is precisely the continuous-time counterpart of so called ``Common Denominator Condition and Change of Variables" (i.e. [Sections~4, 5,  \cite{witsenhausen1988}]), of Witsenhausen's discrete-time stochastic control problems with finite decisions. Witsenhausen in \cite{witsenhausen1988}  called  any discrete-time stochastic dynamical decentralized decision problem which can be transformed via the ``Common Denominator Condition and Change of Variables"  to observations which are not affected by any of the team decisions ``Static". The main point we wish to  make regarding \cite{witsenhausen1988} is the following.
 
 Contrary to the belief in \cite{witsenhausen1988},   and although the distributed observations and information structures of the equivalent stochastic team  decision problem, under the reference probability measure ${\mathbb P}$, are not affected by any of the team decisions, this does not mean than   Marschak's and Radner's  \cite{marschak1955,radner1962,marschak-radner1972} static team theory optimality conditions can be easily  applied to compute the optimal team strategies of  the equivalent team problem. We further elaborate on this point  in Section~\ref{path-i}.   
\end{remark}

The main problem with developing the team and PbP optimality conditions based on variational methods, under the original probability space $\Big(\Omega,{\mathbb F},\{ {\mathbb F}_{0,t}:   t \in [0, T]\}, {\mathbb P}^u\Big)$,  is  the definition of admissible strategies, which states that $\{u_t^i: t \in [0,T]\}$ is adapted to feedback information  $\{ {\cal G}_{0,t}^{I^i}: t \in [0,T]\} \subset \{ {\cal F}_{0,t}^x: t \in [0,T]\}, i=1, \ldots, N$, and hence  affected by the team decisions. Therefore, if one invokes weak or needle variations of $u \in {\mathbb U}^{(N)}[0,T]$, to compute the Gateaux derivative of the pay-off, then one needs the variational equation of the unobserved state $x(\cdot)$ satisfying (\ref{fi10}), which implies that one should  differentiate $\{u_t^i\equiv \mu(t,I^i): t \in [0,T]\}, i=1, \ldots, N$ with respect to $x$, because $\{I^1(t), \ldots, I^N(t): t \in [0,T]\}$ are affected by the decisions.  Therefore, the classical methods which assume nonanticipative strategies  adapted to   $\{ {\cal F}_{0,t}^W: t \in [0,T]\}$ \cite{yong-zhou1999}  or any sub-$\sigma-$algebra of this \cite{ahmed-charalambous2012a}, in general do not apply.  One approach to circumvent this technicality is to show that feedback strategies are dense in noanticipatve or open loop  strategies, and the pay-off is continuously dependent on $u \in {\mathbb U}^{(N)}[0,T]$ as in \cite{charalambous-ahmedFIS_Parti2012}. Another approach is to use Girsanov's theorem.

Before we proceed we show, in the next theorem,  that Girsanov's change of probability measure, which  is based on identifying sufficient conditions so that   $\{ \Lambda^{u}(t): 0 \leq t \leq T\}$  is an $\Big(\{{\mathbb F}_{0,t}: t \in [0,T]\}, {\mathbb P}\Big)$-martingale,  holds under more general conditions than the uniform bounded condition given by Assumptions~\ref{A1-A4}, {\bf (A5)}. 

\begin{theorem} (Equivalence of Dynamic Team Problems)\\
\label{lemma-g}   
Suppose ${\mathbb E} | x(0)|_{{\mathbb R}^n} < \infty$, Assumptions~\ref{A1-A4}, {\bf (A1),  (A2), 
 (A7)}, hold, and consider any of the admissible strategies of Definition~\ref{strategiesf}.\\
Then   
${\mathbb E} \Big(\Lambda^u(t)\Big)=1, \forall t \in [0,T]$, and the dynamic team problem with pay-off (\ref{fi10}) subject to  $x(\cdot)$ satisfying  (\ref{fi9}) is equivalent to the dynamic team problem with pay-off  (\ref{fi6}) with $(x(\cdot), \Lambda(\cdot))$ satisfying (\ref{fi1}), (\ref{fi5}).
\end{theorem}

\begin{proof} See Appendix.

\end{proof}

Thus, Theorem~\ref{lemma-g} is a significant generalization of the equivalence between the two stochastic team problems. 

\subsection{Function Space Integration: Equivalence of Static and Dynamic Team Problems}
\label{path-i}
In this section, we show the precise connection between Girsanov's measure transformation and  Witsenhausen's  ``Common Denominator Condition and Change of Variables" \cite{witsenhausen1988},    for the continuous-time stochastic dynamic team  problem (\ref{pf1})-(\ref{pf4}), and  for general  discrete-time stochastic dynamic team problems.\\

\noi{\bf Continuous-Time Stochastic Dynamic Team Problems.}

For simplicity we introduce the following assumptions.

\begin{assumptions}
\label{ass-drift}
The Borel measurable diffusion coefficient $\sigma$ in (\ref{fi9}) is replaced by\footnote{$G$ can be allowed to depend on $x$.}  

{\bf (A9)} $G: [0,T] \longrightarrow {\cal L}({\mathbb R}^n, {\mathbb R}^n)$ (i.e. it is independent of $(x,u) \in {\mathbb R}^n \times {\mathbb A}^{(N)}$), $G^{-1}$ exists and both are  uniformly bounded;

{\bf (A10)} ${\mathbb E}|x(0)|_{{\mathbb R}^n} < \infty$ and Assumptions~\ref{A1-A4}, {\bf (A1)} hold. 


\end{assumptions}

Under Assumptions~\ref{ass-drift}, by Theorem~\ref{lemma-g} we can apply  Girsanov's theorem to obtain  the equivalent stochastic dynamic team problem under the   probability space $\Big(\Omega, {\mathbb F}, \{ {\mathbb F}_{0,t}: t \in [0,T]\}, {{\mathbb P}}\Big)$, such that  $\{(x(t), \Lambda^u(t): t \in [0,T]\}$ is  defined by  
\begin{align}
x(t) =& x(0)+ \int_{0}^t G(s) dW(s) \equiv x(0) + \widehat{W}(t), \nonumber \\
 \Lambda^u(t) =&1+ \int_{0}^t f^*(s,x(s), u_s) \Big(G(s) G^{*}(s)\Big)^{-1}dx(s),   \label{path4}
\end{align}
where $\{ W(t): t \in [0,T]\}$ is a standard Brownian motions, and for any $u \in {\mathbb U}^{(N)}[0,T]$ the  the pay-off givan  is
 \begin{align} 
  J(u) \tri
    {\mathbb E} \biggl\{   \int_{0}^{T}   \Lambda^u(t)  \ell(t,x(t),u_t) dt +  \Lambda^u(T)  \varphi(x(T)\biggr\}.  \label{path5}
  \end{align} 
Note that in (\ref{path5}),  ${\mathbb E}\{\cdot\}$ denotes expectation with respect to the product measure ${\mathbb P}(d\xi,dw)\tri \Pi_{0}(d\xi)\times {\cal W}_{\widehat{W}}(dw)$, where ${\cal W}_{\widehat{W}}(\cdot)$ is the Wiener measure on the Brownian motion sample paths  $\{\widehat{W}(t): t \in [0,T]\} \in C([0,T], {\mathbb R}^n)$. \\
Now, we consider the equivalent transformed pay-off (\ref{path5}), and integrate by parts the stochastic integral term appearing in $\Lambda^u(\cdot)$, and then  we define
\bea
\overline{\ell}(t,\xi,\widehat{W}, u)\tri  \Lambda^u \ell(t,x,u), \hso \overline{\varphi}(T,\xi,\widehat{W}, u) \tri \Lambda^u  \varphi(x). \label{path6}
\eea
Then the transformed pay-off (\ref{path5})  is given by
\begin{align} 
&  J(u) \tri
    \int_{{\mathbb R}^n\times C([0,T], {\mathbb R}^n) } \biggl\{   \int_{0}^{T}   \overline{\ell}(t,\xi,\{\widehat{W}(s),u_s: 0\leq s \leq t\}) dt \nonumber \\
  &  +  \overline{\varphi}(T,\xi,\{\widehat{W}(t), u_t: 0 \leq t \leq T\}) \biggr\}  {\cal W}_{\widehat{W}}(d\widehat{W}) \times \Pi_0(d\xi)  \label{path7}
  \end{align} 
Note that the equivalent pay-off (\ref{path7}) is expressed as a function space integral with respect to a Wiener measure. Such function space integrations  are  discussed in  nonlinear mean-square error nonlinear filtering  problems in \cite{benes1981,charalambous-elliott1998}. Moreover, expression (\ref{path7}) is precisely the continuous-time analog of Witsenhausen's  main theorem [Theorem~6.1, \cite{witsenhausen1988}], which is easily verified by comparing (\ref{path7}) and [equation (6.4),  \cite{witsenhausen1988}].  In fact, $\Lambda^u(\cdot)$   is the common denominator condition, and the representation of the pay-off as a functional of $\omega(t)=(x(0), W(t)\tri x(0)+ \widehat{W}(t)), t \in [0,T]$, is the change of variables. Since (\ref{path7}) is a functional of $(x(0), \widehat{W}(\cdot), u^1, \ldots, u^N)$ and  $u^i$'s are not affected by any of the team decisions, because $u_t^i = \mu_t^i(I^i)$, and  $I^i(\cdot)$  are functionals of $(x(0), \widehat{W}(\cdot))$, then we  can proceed further to derive team optimality conditions using static team theory, by  computing the Gateaux derivative  of (\ref{path7}) at $u^{i,o} $ in  the direction of $u^i-u^{i,o}, i=1, \ldots$, as in  \cite{radner1962,krainak-speyer-marcus1982a}. However, for complicated problems  this procedure might not be  tractable even for the simplified case  $\ell=0$, because it will involve function space integrations with respect to the Wiener measure. \\

\noi{\bf Discrete-Time Stochastic Dynamic Team Problems.}

Next,  we consider a discrete-time generalized version  of the stochastic differential decentralized decision problem (\ref{path4}), (\ref{path5}), under the reference probability measure ${\mathbb P}$.   Let ${\mathbb N}_0 \tri \{0, 1, 2, \ldots\}, {\mathbb N}_1 \tri \{1, 2, \ldots\} $ denote time-index sets.

 We start with  a reference probability space $\Big(\Omega, {\mathbb F}, \{ {\mathbb F}_{0,n}: n \in {\mathbb N}_0\}, {{\mathbb P}}\Big)$, under which $\{x(n): n \in {\mathbb N}_0\}$ is a sequences of independent RVs, having Normal densities denoted by $\lambda_n (\cdot) \sim {\mathbb N}(0, G(n)G^*(n))$, for $ n \in {\mathbb N}_1$,  $x(0) \sim \Pi_0(dx)$,  $\{{\mathbb F}_{0,n}: n \in {\mathbb N}_0\}$ is the  filtration  generated by the completion of the $\sigma-$algebra $\sigma\{x(k): k \leq n\}, n \in {\mathbb N}_0$, and  $\{{\cal G}_{0,n}^{z^i}: n \in {\mathbb N}_0\}$ is the  filtration  generated by the completion of the  $\sigma-$algebra $\sigma\{z^i(k) \tri h^i(k, x(0), x(1), \ldots, x(k)): k \leq n\}, n \in {\mathbb N}_0, i=1, \ldots, N$. \\
Next, we define the team strategies which donot assume ``Perfect Recall".  Suppose  for each $n \in {\mathbb N}_0$, $u^i(n) \in {\mathbb A}_n^i$, and that $u^i(n)$ is measurable with respect to ${\cal E}_n^i \subset  \bigvee_{i=1}^N {\cal G}_{0,n}^{z^i}$, where ${\cal E}_n^i$ is not  nested, for $i=1, \ldots, N$, that is, ${\cal E}_n^i \nsubseteq  {\cal E}_{n+1}^i, n =0, 1, \ldots$, and hence  all control station donot have perfect recall. For each $n \in {\mathbb N}_0$,  such information structures can be generated by   ${\cal E}_n^i \tri \sigma\Big\{ {\Pi }_n^i\Big(\{z^j(0), z^j(1), \ldots, z^j(n): j=1, \ldots, N\}\Big)\Big\}$, where ${\Pi}_n^i(\cdot)$ is the projection   to a subset of $\{z^j(0), z^j(1), \ldots, z^j(n): j =1, \ldots, N\}$,  for $ i=1, \ldots, N$. \\
  We denote the set of admissible strategies at the $i$th control station  at time $n \in {\mathbb N}_0$, by $\gamma_n^i(\cdot) \in {\mathbb U}^i[n]$,  their $T-$tuple by $\gamma_{[0,T-1]}^i (\cdot)\tri (\gamma_0^i(\cdot), \ldots, \gamma_{T-1}^i(\cdot)) \in {\mathbb U}^i[0,T-1] \tri \times_{j=0}^{T-1} {\mathbb U}^i[j]$, $i=1, \ldots, N$, and  $\gamma_{[0,T-1]}(\cdot) \tri (\gamma_{[0,T-1]}^1, \ldots, \gamma_{[0,T-1]}^N(\cdot) \in {\mathbb U}^{(N)}[0,T-1] \tri \times_{i=1}^N {\mathbb U}^i[0,T-1]$.

Consider the following measurable functions.
\begin{align}
&f(k, \cdot, \cdot):  \times_{i=0}^k ({\mathbb R}^n) \times_{i=1,j=0}^{N,k} ( {\mathbb A}_j^i) \longrightarrow {\mathbb R}^n, \hso k \in {\mathbb N}_0, \nonumber \\
& h^i(k,\cdot): \times_{i=0}^k( {\mathbb R}^n) \longrightarrow {\mathbb R}^{k_i}, \hso k \in {\mathbb N}_0, \hso i=1, \ldots, N . \nonumber 
\end{align}
  For any admissible decentralized strategy   $u \equiv \gamma_{[0,T-1]}  \in {\mathbb U}^{(N)}[0,T-1]$, we define the following quantity.

\begin{align}
\Lambda_{0,n+1}^u\tri & \prod_{k=0}^n    \frac{\lambda_{k+1}(x(k+1)-f(k,x(0), x(1), \ldots, x(k),u(k)))}{\lambda_{k+1}(x(k+1))},  \nonumber \\
 \Lambda_{0,0}^u=&1, \; n \in {\mathbb Z}_+. \label{dt1} 
\end{align}
Under the reference probability space $\Big(\Omega, {\mathbb F}, \{ {\mathbb F}_{0,n}: n \in {\mathbb N}_0\}, {{\mathbb P}}\Big)$, define the team pay-off 
\begin{align}
J(u)=& {\mathbb E} \Big\{  \Lambda_{0,T}^u(x(0),u(0),\ldots, x(T-1), u(T-1), x(T)) \Big(\sum_{k=0}^{T-1} \ell(k,x(k),u(k)) + \varphi(x(T)) \Big) \Big\} \label{dt6a} \\
=& \int \Big\{\Lambda_{0,T}^u(x(0),u(0),\ldots, x(T-1), u(T-1), x(T)) \Big(\sum_{k=0}^{T-1} \ell(k,x(k),u(k)) + \varphi(x(T)) \Big) \Big\}\nonumber \\
&  \prod_{k=0}^{T-1} \lambda_{k+1}(x(k+1)) dx(k+1)  \Pi_{0}(dx(0))  \label{dt7}  \\
& \equiv  \int L(\gamma_{[0,T-1]}, x(0), x(1), \ldots, x(T)) . \prod_{k=0}^{T-1} \lambda_{k+1}(x(k+1)) dx(k+1)  \Pi_{0}(dx(0)) \equiv J(\gamma_{[0,T-1]}) . \label{stdt7}
\end{align}
Clearly,  the problem 
\bea
\inf \Big\{  J(\gamma_{[0,T-1]}): \gamma_{[0,T-1]}  \in {\mathbb U}^{(N)}[0,T-1]\Big\}, \label{stp8}
\eea
 is a static team problem, because the information structure available for decisions are nonlinear measurable functions of $(x(0), \ldots, x(T))$, which are not affected by any of the team decisions. This is the transformed equivalent stochastic team problem of a certain stochastic dynamic team problem which we introduce next.  

Since it can be shown that $\{\Lambda_{0,n}^u: n \in {\mathbb N}_0\}$ is an $\Big(\Omega, {\mathbb F}, \{ {\mathbb F}_{0,n}: n \in {\mathbb N}_0\}, {{\mathbb P}}\Big)-$martingale, with $\int  \Lambda_{0,n}^u(\omega) d{\mathbb P}(\omega)=1$, then  we can define the  probability measure ${\mathbb P}^u$ on  $\Big(\Omega, \{ {\mathbb F}_{0,n}: n \in {\mathbb N}_0\}\Big)$ by setting
\begin{align}
\frac{ d {\mathbb P}^u}{ d {{\mathbb P}}} \Big{|}_{ {\mathbb F}_{0,n}} = \Lambda^u(n), \hst \forall n \in {\mathbb N}_0. 
 \label{dt2}
\end{align} 
Then under this  probability measure ${\mathbb P}^u$, the process defined by
\begin{align}
w^u(n+1) \tri x(n+1) - f(n,x(0), \ldots, x(n),u(n)),\;  n \in {\mathbb N}_0, \label{dt3}
\end{align}
is a sequences of independent normally distributed RVs with densities, $\lambda_n(\cdot), n \in {\mathbb N}_1$. Therefore, under the probability space $\Big(\Omega, {\mathbb F}, \{ {\mathbb F}_{0,n}: n \in {\mathbb Z}_+\}, {{\mathbb P}^u}\Big)$, we have 
\begin{align}
 x(n+1) =& f(n,x(0), \ldots, x(n),u(n))+ w^u(n+1), \label{dt4} \\
 z^i(n)= & h^i(n,x(0), \ldots, x(n)), \hso n \in {\mathbb N}_0, i=1, \ldots, N. \label{dt5} 
\end{align}
Then, for any admissible discrete-time team strategy $u \equiv \gamma_{[0,T-1]} \in {\mathbb U}^{N}[0,T-1]$, under measure ${\mathbb P}^u$ the team pay-off is 
\begin{align}
J(u)= {\mathbb E}^u \Big\{ \sum_{k=0}^{T-1} \ell(k,x(k),u(k)) + \varphi(x(T))\Big\}. \label{dt6}
\end{align}
Thus, we have shown that the dynamic team problem of minimizing  pay-off (\ref{dt6}), subject to  (\ref{dt4}), (\ref{dt5}) can be transformed to the equivalent static team problem defined by  pay-off (\ref{stp8}),   where $\{x(n): n \in {\mathbb N}_0\}$ is  an  independent sequence, distributed according to  $x(0) \sim \Pi_0(\cdot), \{ \lambda_{n+1}(\cdot): n \in {\mathbb N}_0\}$,  and the information structures  are  functions of the independent sequence $\{x(n): n \in {\mathbb N}_0\}$. \\
Consequently, we have identified the precise connection between the so-called ``Common Denominator Conditions and Change of Variables" described by Witsenhausen in \cite{witsenhausen1988}), for a discrete-time stochastic dynamic team problem to be equivalent to a static team problem.    

Therefore, we conclude that the static team theory by Marschak and Radner \cite{marschak1955,radner1962} and its generalization in \cite{krainak-speyer-marcus1982a}, are directly applicable to the transformed problem (\ref{stp8}),  in a higher dimension consisting of $TN$ decision strategies (because we have assumed no perfect recall of information at all control stations). Thus, any of the  theorems found in \cite{krainak-speyer-marcus1982a}  are applicable.

Finally, we make the following observations.

\begin{remark}
\label{static-rem}
The Girsanov theorem is also applicable to more general models that (\ref{dt4}) such as,  $x(n+1)=f(n, x(0), \ldots, x(n), u(n), w(n+1))$, where $\{w(n): n \in {\mathbb N}_1\}$ is arbitrary distributed, and  to  finite and countable state Markov decision models.  It is also applicable to stochastic dynamic team games with noisy information such as, $z^i(n)=h^i(n, x(0), \ldots, x(n), u(n), v^i(n))$, where $\{v^i(n): n \in {\mathbb N}_0\}, i=1, \ldots, N$ are arbitrary distributed. 
\end{remark}

   \subsection{Continuous Dependence of Solutions and Semi Martingale Representation}
   \label{dif-sem}
  In this section, we show under appropriate conditions, that (\ref{fi5})  has  unique continuous solutions (in the strong sense)  with  finite second moments,  and that any solution  is continuously dependent on $u$. These properties are  required  in the derivation of necessary conditions for team and PbP  optimality of the  equivalent  transformed  team problem.

\begin{lemma}(Existence and Differentiability)
\label{lemma3.1}
Suppose Assumptions~\ref{A1-A4}, {\bf (A2)-(A5)}  hold.  Then for any ${\mathbb  F}_{0,0}$-measurable initial state $x_0$ having finite second moment, and any $u \in {\mathbb U}^{(N)}[0,T]$,  the following hold.

 
{\bf (1)} (\ref{fi1}),  (\ref{fi5}) have unique solutions   $x \in B_{{\mathbb F}_T}^{\infty}([0,T],L^2(\Omega,{\mathbb R}^n))$,     $\Lambda^u \in B_{{\mathbb F}_T}^{\infty}([0,T],L^2(\Omega,{\mathbb R}))$  having a continuous modifications, that is, $\Lambda^u \in C([0,T],{\mathbb R})$,  ${\mathbb P}-$a.s. Moreover, $\Lambda^u \in L^p(\Omega, \{{\mathbb F}_{0,t}: t \in [0,T]\}, {\mathbb P}; {\mathbb R})$ for any finite $p$, and also $\Lambda^u \in L^\infty(\Omega,\{ {\mathbb F}_{0,t}: t\in [0,T]\}, {\mathbb P}; {\mathbb R})$

{\bf (2)}  Under the additional Assumptions~\ref{A1-A4}, {\bf (A8)},  the solution of    (\ref{fi5}) is continuously dependent on the decisions, in the sense that, as $u^{i, \alpha} \longrightarrow u^{i,o}$  in ${\mathbb U}^{I^i}[0,T]$, $\forall i \in {\mathbb Z}_N$,  $\Lambda^\alpha \buildrel s \over\longrightarrow \Lambda^o $ in $B_{{\mathbb F}_T}^{\infty}([0,T],L^2(\Omega,{\mathbb R}))$.

\end{lemma} 

\begin{proof} {\bf (1)}.  The uniqueness of solution having a continuous modification is  already discussed in Section~\ref{full}, and it is based on Assumptions~\ref{A1-A4}, {\bf (A2)-(A5)}. The rest of the claims  are shown by following the method in \cite{charalambous-ahmedPISG2013a}.

\noi {\bf (2)}  Next, we consider the second part asserting the  continuity of $u$ to solution map  $u\longrightarrow \Lambda^u$  Let $\{\{u^{i,\alpha}: i=1,2,\ldots, N\},u^o\}$ be any pair of strategies from ${\mathbb U}^{(N)}[0,T] \times {\mathbb U}^{(N)}[0,T]$ and  $\{\Lambda^\alpha, \Lambda^o\}$ denote the corresponding pair of solutions of (\ref{fi5}).  Let $u^{i,\alpha} \longrightarrow u^{i,o}$ in ${\mathbb U}^i[0,T], i=1,2,\ldots,N$.  We must show that  $\Lambda^\alpha  \longrightarrow \Lambda^o$ in  $B_{{\mathbb F}_T}^{\infty}([0,T], L^2(\Omega,{\mathbb R})).$    By the definition of solution to (\ref{fi5}), we have
\begin{align} 
&\Lambda^\alpha(t)-\Lambda^o(t)  = \int_0^t       \Big\{\Lambda^{\alpha}(s) -\Lambda^o(s) \Big\} f^*(s,x(s),u_s^o) a^{-1}(s,x(s))dx(s) \nonumber  \\
&+ \int_0^t  \Lambda^{\alpha}(s) \Big\{f^*(s,x(s),u_s^\alpha)-f^*(s,x(s),u_s^o)\Big\}a^{-1}(s, x(s))  dx(s),
 \label{gs20}
\end{align}
where $a \tri \sigma \sigma^*$, and $\{x(t): t \in [0, T]\}$ is the solution of (\ref{fi1}). 
From (\ref{gs20}) using Doobs martingale inequality we obtain 
\begin{align}
&{\mathbb E} | \Lambda^\alpha(t)-\Lambda^o(t)|^2  \leq   4  {\mathbb E} \ \int_0^t      |\Lambda^{\alpha}(s) -\Lambda^o(s)|^2  |\sigma^{-1}(s, x(s))f(s,x(s),u_s^o)|_{{\mathbb R}^n}^2 ds \nonumber  \\
&+  4  {\mathbb  E} \int_0^t   |\Lambda^{\alpha}(s)|^2 |\sigma^{-1}(s, x(s)) (f(s,x(s),u_s^\alpha(s))-f(s,x(s),u_s^o))|_{{\mathbb R}^n}^2 ds
 \label{gs20a}  \\
   &\leq  4 K_1  {\mathbb E} \ \int_0^t      |\Lambda^{\alpha}(s) -\Lambda^o(s)|^2  ds +  4 K_2  {\mathbb  E} \int_0^t   |\Lambda^{\alpha}(s)|^2 |u_s^\alpha-u_s^o|_{{\mathbb R}^d}^2 ds,  \hst t \in [0,T], \label{gs200a}  
 \end{align}
 where (\ref{gs200a}) follows from {\bf (A5), (A8)}, with $K_1, K_2>0$. \\
Since by part {\bf (1)},  ${\mathbb P}- ess \sup_{\omega \in \Omega} \sup_{t \in [0,T]}|\Lambda^\alpha(t,\omega)|^2 < M$ for some  finite $M>0$,  applying this in (\ref{gs200a}) we deduce the following bound.
\begin{align}  
E|\Lambda^\alpha(t)-\Lambda^o(t)|^2   \leq  4 K_1 {\mathbb E} \ \int_0^t      |\Lambda^{\alpha}(s) -\Lambda^o(s)|^2   ds +  4M K_2  {\mathbb  E} \int_0^t  |u_s^\alpha-u_s^o|_{{\mathbb R}^d}^2 ds,  \hst t \in [0,T]. \label{gs20fff}
\end{align} 
Now,   letting $\alpha \longrightarrow \infty$ and recalling that $u^{i,\alpha} \longrightarrow u^{i,o}$ in ${\mathbb U}^i[0,T]$ the integrand in the second  right hand side of (\ref{gs20fff}) converges to zero for almost all $s \in [0,T], {\mathbb P}-$a.s. Since by our assumptions the integrands are dominated by integrable functions,   we obtain  by Gronwall's inequality that $\lim_{\alpha \longrightarrow \infty}  \sup_{t \in [0,T]}
{\mathbb E} |\Lambda^\alpha(t)-\Lambda^o(t)|^2  = 0$. This completes the derivation. 
\end{proof}

\ \

The derivation of stochastic minimum principle  will be based on    certain fundamental   properties of semi martingales on Hilbert spaces,  which we describe below.

\begin{definition}
\label{definition4.2}
  An ${\mathbb R}^{n}-$valued  random process $\{m(t): t \in [0,T]\}$ is said to be a square integrable continuous   $\{{\mathbb F}_{0,t}: t \in [0,T]\}-$semi martingale if and only if  it has a representation  
  \begin{eqnarray}
 m(t) = m(0) + \int_0^t v(s) ds + \int_0^t \Sigma(s) dW(s), \; t \in [0,T], \label{eq12}
  \end{eqnarray} 
   for some $v \in L_{{\mathbb F}_T}^2([0,T],{\mathbb R}^n)$ and $\Sigma \in L_{{\mathbb F}_T}^2([0,T],{\cal L}({\mathbb R}^m,{\mathbb R}^n))$  and for some ${\mathbb R}^n-$valued ${\mathbb  F}_{0,0}-$measurable  random variable  $m(0)$ having finite second moment. The set of all such semi martingales is denoted bu ${\cal S M}^2[0,T]$. 
\end{definition}

 \noi Introduce the following class of $\{{\mathbb F}_{0,t}: t \i [0,T]\}-$semi martingales:
   \begin{align}
  & {\cal SM}_0^2[0,T] \tri \Big\{ m:   m(t)   = \int_0^t v(s) ds + \int_0^t \Sigma(s) dW(s),  \nonumber \\
   &   t \in [0,T], \:
      \mbox{for} \; v \in L_{{\mathbb F}_T}^2([0,T],{\mathbb R}^n), \Sigma \in L_{{\mathbb F}_T}^2([0,T],{\cal L}({\mathbb R}^m,{\mathbb R}^n)) \Big\}.
 \nonumber 
    \end{align}

\noi Now we present the fundamental result which is utilized in the maximum principle derivation.

\begin{theorem}(Semi Martingale Representation)
\label{theorem4.3} 
The class of semi martingales ${\cal SM}_0^2[0,T]$ is a real linear vector space and it is a  Hilbert space with respect to the norm topology   $\parallel m\parallel_{{\cal SM}_0^2[0,T]}$ arising from
\bes
 \parallel m \parallel_{{\cal SM}^2_0[0,T]}^2  \tri {\mathbb  E}\int_{[0,T]} |v(t)|_{{\mathbb R}^n}^2 dt + {\mathbb  E} \int_{[0,T]} tr(\Sigma^*(t)\Sigma(t)) dt. 
 \ees
Moreover,  the space ${\cal SM}_0^2[0,T]$ is isometrically isomorphic to the space $L_{{\mathbb F}_T}^2([0,T],{\mathbb R}^n)\times L_{{\mathbb F}_T}^2([0,T],{\cal L}({\mathbb R}^m,{\mathbb R}^n)).$
\end{theorem}

\begin{proof}  This is found in many books.

\end{proof}

\section{Dynamic Team Optimality Conditions }
\label{optimality}
In this section we derive the  team and PbP optimality conditions, under the reference probability measure ${\mathbb P}$, and then we translate the results under the original probability measure ${\mathbb P}^u$. For the  derivation of stochastic  optimality conditions  we shall require stronger regularity conditions.
 These  are given below.

\begin{assumptions}
\label{NC1}  
 ${\mathbb A}^i$ is a closed, bounded and convex subset of ${\mathbb R}^{d_i},  \forall i \in {\mathbb Z}_N$, ${\mathbb E}| x(0)|_{{\mathbb R}^n}^2 < \infty$,     the maps  $\{f,\sigma, \ell, \varphi\}$ are Borel measurable, $\{h^i: i=1, \ldots, N\}$ are  progressively  measurable, defined by 
\begin{align}
& f: [0,T] \times {\mathbb R}^n \times {\mathbb A}^{(N)} \longrightarrow {\mathbb R}^n , \hso \sigma: [0,T] \times {\mathbb R}^n  \longrightarrow {\cal L}({\mathbb R}^n, {\mathbb R}^n), \nonumber \\
& \varphi:  {\mathbb R}^n  \longrightarrow R, \hso
 \ell : [0,T] \times {\mathbb R}^n \times {\mathbb A}^{(N)} \longrightarrow {\mathbb R} , \hso h^i:[0,T] \times C([0,T], {\mathbb R}^n)  \longrightarrow {\mathbb R}^{k_i}, \nonumber 
 \end{align}
and they satisfy the following conditions.


{\bf (C1)} The map $\sigma$ satisfies {\bf (A2), (A3), (A4)} and the map $\sigma^{-1}f$ satisfies {\bf (A5)};

{\bf (C2)} The map $f $ is once continuously differentiable with respect to $u \in {\mathbb A}^{(N)}$, and  the first derivative of $\sigma^{-1}f$ with respect to is $u$ is bounded  uniformly in $(t, x, u) \in [0, T] \times {\mathbb R}^n \times {\mathbb A}^{(N)}$;

{\bf (C3)} The maps   $\ell$ is once continuously differentiable with respect to $u \in {\mathbb A}^{(N))}$,  and there exists a $K>0$ such that 
\begin{align}
&\Big(1+ |x|_{{\mathbb R}^n}^2 +|u|_{{\mathbb R}^d}^2\Big)^{-1} |\ell(t,x,u)|_{{\mathbb R}} + \Big(1+ |x|_{{\mathbb R}^n} +|u|_{{\mathbb R}^d}\Big)^{-1}   |\ell_u (t,x,u)|_{{\mathbb R}^d} \leq K, \nonumber \\
 &\Big(1+ |x|_{{\mathbb R}^n}^2|\Big)^{-1} | \varphi(x)|_{{\mathbb R}}\leq K;  \nonumber 
\end{align}

 \item
\end{assumptions}


\subsection{Necessary Conditions for Team Optimality}
\label{ncto}

Next, we prepare to give the variational equation under the reference measure probability $\Big(\Omega ,{\mathbb F},\{ {\mathbb F}_{0,t}: t \in [0,T]\}, {\mathbb P}\Big)$. We define the Gateaux derivative of any matrix valued function $G : [0,T] \times {\mathbb R}^n \times {\mathbb A}^{(N)} \longrightarrow {\cal L}({\mathbb R}^n, {\mathbb R}^m)$ with respect to the  variable  at the point $(t,x,u) \in [0,T] \times {\mathbb R}^{n}\times  {\mathbb A}^{(N)}$   in the direction $v \in {\mathbb A}^{(N)}$  by
\bes
   G_u(t,x,u; v) \tri   \lim_{\varepsilon \longrightarrow 0}\frac{1}{\varepsilon} \Big\{ G(t,x,u + \varepsilon v)- G(t,x,u)\Big\}, t \in [0,T].
  \ees 
Clearly, for each column of $G$ denoted by $G^{(j)}, j=1, \ldots,n$, the Gateaux derivative of $G^{(j)}$ component wise is given by $G_u^{(j)}(t,x,u;v) = G_u^{(j)} (t,x,u)v, t \in [0,T]$.\\
Suppose $u^o \tri (u^{1,o}, u^{2,o}, \ldots, u^{N,o}) \in {\mathbb U}^{(N)}[0,T]$ denotes the optimal decision and $u \tri (u^1, u^2, \ldots, u^N) \in {\mathbb U}^{(N)}[0,T]$ any other decision.  Since ${\mathbb U}^{I^i}[0,T]$ is convex $\forall i \in {\mathbb Z}_N$, it is clear that  for any $\varepsilon \in [0,1]$, 
\bes
 u_t^{i,\varepsilon} \tri u_t^{i,o} + \varepsilon (u_t^i-u_t^{i,o}) \in {\mathbb U}^{I^i}[0,T], \hst \forall i \in {\mathbb Z}_N.
 \ees
 Let $\Lambda^{\varepsilon}(\cdot)\equiv \Lambda^\veps(\cdot; u^\veps(\cdot))$ and  $\Lambda^{o}(\cdot) \equiv \Lambda^o(\cdot;u^o(\cdot))  \in B_{{\mathbb F}_T}^{\infty}([0,T],L^2(\Omega,{\mathbb R}))$ denote the solutions  of the differential system  (\ref{fi5})  corresponding to  $u^{\varepsilon}(\cdot)$ and $u^o(\cdot)$, respectively.  Consider the limit 
 \bes
  Z(t) \tri \lim_{\varepsilon\longrightarrow 0}  \frac{1}{\veps} \Big\{\Lambda^{\varepsilon}(t)-\Lambda^o(t)\Big\} , \hst t \in [0,T]. 
  \ees
  Since under the reference probability measure, $dx(t)=\sigma(t,x(t))dW(t), x(0)=x_0$ then its solution is not affected by $u^\eps$, hence we do not consider any variations of $x(\cdot)$ with respect to $u^\eps$. \\ 
We have the following  characterization of  the variational process $\{Z(t): t \in [0,T]\}$.

\begin{lemma}(Variational Equation)
\label {lemma4.1f}
Suppose Assumptions~\ref{NC1} hold. The process $\{Z(t): t \in [0,T]\}$ is an element of the Banach space    $B_{{\mathbb F}_T}^{\infty}([0,T],L^2(\Omega,{\mathbb R}))$ and it  is the unique solution of the variational stochastic differential equation
 \begin{align} 
 &dZ(t) = f^{*}(t,x(t),u_t^o) \sigma^{*,-1}(t,x(t)) Z(t) dW(t) \nonumber \\
 &+ \sum_{i=1}^N   (f^* \sigma^{*,-1})_{u^i}(t,x(t),u_t^{o}; u_t^i-u_t^{i,o}) dW(t), \hst Z(0)=0. \label{eq9f}   
 \end{align} 
 having a continuous modification.  
  \end{lemma}

\begin{proof} This follows directly from Lemma~\ref{lemma3.1} and the fact that $f^* \sigma^{*,-1}$ and its derivative with respect to $u$ are uniformly bounded.
\end{proof} 

\ \

\noi {\bf Minimum Principle Under  Reference Probability Space $\Big( {\Omega}, {\mathbb F},  \{ {\mathbb F}_{0,t}: t \in [0,T]\}, {\mathbb P}\Big)$}\\
Next, we state the necessary conditions for team and PbP optimality under the reference probability measure.

Define the Hamiltonian of the augmented system  (\ref{fi1}),  (\ref{fi5}), (\ref{fi6}). 
\begin{align}
 {\cal  H}: [0, T]  &\times  {\mathbb R}^{n} \times   {\mathbb R} \times {\cal L}( {\mathbb R}^{n}, {\mathbb R})\times {\mathbb A}^{(N)} \longrightarrow {\mathbb R} \nonumber \\
    {\cal H} (t,x,\Lambda, \Psi,Q, u)  \tri  &  \Lambda  Q \sigma^{-1}(t,x)f(t,x,u) + \Lambda \ell(t,x,u),  \hst  t \in  [0, T]. \label{fs1}
    \end{align}
For any $u \in {\mathbb U}^{(N)}[0,T]$, the adjoint process $(\Psi, Q) \in    
L_{{\mathbb F}_T}^2([0,T], {\mathbb R}) \times L_{{\mathbb F}_T}^2([0,T],  {\cal L}( {\mathbb R}^{n}, {\mathbb R}))$  satisfies the following backward stochastic differential equations
\begin{align} 
d\Psi (t)  =&  - \ell(t,x(t),u_t) dt - Q(t) \sigma^{-1}(t,x(t))f(t,x(t), u_t) 
+ Q(t)  dW(t),    \nonumber    \\ 
=&- {\cal H}_\Lambda (t,x(t), \Lambda(t), \Psi(t),Q(t),u_t) dt  + Q(t)  dW(t), \hso t \in [0,T),  \; \Psi(T)= \varphi (x(T)) , \label{fs2}  
 \end{align}
 The state process satisfies the stochastic differential equation (\ref{fi5}) expressed in terms of the Hamiltonian as follows.
\begin{align}
&d \Lambda (t)         =\Lambda(t) f^*(t,x(t), u_t) \sigma^{*,-1}(t,x(t))dW(t), \nonumber \\
= &\Lambda(t) f^*(t,x(t), u_t) \sigma^{*,-1}(t,x(t)) dW(t), \; t \in (0, T], \;
        \Lambda(0) =  1. \label{fs3}  
 \end{align}
Moreover, under measure ${\mathbb P}$, the process  $\{x(t): t \in [0, T]\}$ is not affected by $u \in {\mathbb U}^{(N)}[0, T]$ and satisfies
\bea
dx(t)=\sigma(t,x(t))dW(t), \hst t \in (0, T], \hso x(0)=x_0 \label{spu1}
\eea
Next, we state the  the necessary conditions for an element $u^o \in {\mathbb U}^{(N)}[0,T]$  to be team optimal.

\begin{theorem}(Necessary Conditions for Team Optimality under Reference Measure)\\
\label{thmfs1}
Suppose Assumptions~\ref{NC1} hold. Then we have the following..
 
\noi {\bf Necessary Conditions.}    For  an element $ u^o \in {\mathbb U}^{(N)}[0,T]$ with the corresponding solution $\Lambda^o \in B_{{\mathbb F}_T}^{\infty}([0,T], L^2(\Omega,{\mathbb R}))$ to be team optimal, it is necessary  that 
the following hold.


{\bf (1)}  There exists a semi martingale  $m^o \in {\cal SM}_0^2[0,T]$ ($1$-dimensional) with the intensity process $({\Psi}^o,Q^o) \in  L_{{\mathbb F}_T}^2([0,T],{\mathbb R})\times L_{{\mathbb F}_T}^2([0,T],{\cal L}({\mathbb R}^{n},{\mathbb R}))$.
 
{\bf (2) }  The variational inequalities are satisfied:
\begin{align}   
 & \sum_{i=1}^N {\mathbb  E} \Big\{ \int_0^T    {\cal H} (t,x(t), \Lambda^o(t), \Psi^o(t), Q^{o}(t), u_t^{-i,o},u_t^i) dt \Big\}  \nonumber \\
     &\geq  \sum_{i=1}^N {\mathbb  E} \Big\{ \int_0^T    {\cal H} (t,x(t), \Lambda^o(t), \Psi^o(t), Q^{o}(t), u_t^{-i,o}, u_t^{i,o})) dt \Big\} ,   \hso \forall u \in {\mathbb U}^{(N)}[0,T], \label{fs4}
     \end{align}
     \begin{align}
  {\mathbb  E} \Big\{ \int_0^T &   {\cal H} (t,x(t), \Lambda^o(t), \Psi^o(t), Q^{o}(t), u_t^{-i,o},u_t^i) dt \Big\}  \nonumber \\
     \geq  {\mathbb  E} \Big\{ \int_0^T &   {\cal H} (t,x(t), \Lambda^o(t), \Psi^o(t), Q^{o}(t), u_t^{-i,o}, u_t^{i,o})) dt \Big\} ,    \hso  \forall u^i \in {\mathbb U}^{I^i}[0,T], \hso \forall i \in {\mathbb Z}_N. \label{fs4si}
\end{align}

{\bf (3)}  The process $({\Psi}^o,Q^o) \in  L_{{\mathbb F}_T}^2([0,T],{\mathbb R})\times L_{{\mathbb F}_T}^2([0,T],{\cal L}({\mathbb R}^{n}, {\mathbb R}))$ is a unique solution of the backward stochastic differential equation (\ref{fs2}) such that $u^o \in {\mathbb U}^{(N)}[0,T]$ satisfies  the  point wise almost sure inequalities with respect to the $\sigma$-algebras ${\cal G}_{0,t}^{I^i}   \subset {\mathbb F}_{0,t}$, $ t\in [0, T], i=1, 2, \ldots, N:$ 

\begin{align} 
 {\mathbb E} &\Big\{   {\cal H}(t,x(t), \Lambda^o(t),  \Psi^o(t),Q^o(t),u_t^{-i,o},u_t^i)   |{\cal G}_{0, t}^{I^i} \Big\}   \nonumber \\
 & \geq  {\mathbb E} \Big\{   {\cal H}(t,x(t), \Lambda^o(t),  \Psi^o(t),Q^o(t),u_t^{o})   |{\cal G}_{0, t}^{I^i} \Big\}, \hso \forall u^i \in {\mathbb A}^i,  a.e. t \in [0,T], {\mathbb P}|_{{\cal G}_{0,t}^{I^i}}- a.s., \forall i \in {\mathbb Z}_N.   \label{fs12} 
\end{align}

{\bf (4)} For admissible strategies ${\mathbb U}^{(N), z}[0, T], {\mathbb U}^{(N), x}[0, T]$ the conditional expectation in (\ref{fs12}) is taken with respect to the information structures ${\cal G}_{0,t}^{z^i}, {\cal G}^{x^i(t)}$, respectively.

\end{theorem}

\begin{proof}  The derivation is based on the variation equation of Lemma~\ref{lemma4.1f},  the semi martingale representation theorem, and the Riesz representation theorem. We outline the steps. By Assumptions~\ref{NC1}, it can be shown that the Gateaux derivative of $J(\cdot)$ at $u^o$ in the direction $u-u^o$ exists, and  it is computed via   $\frac{d}{d \eps} J(u^o +\eps (u-u^o))|_{\eps =0}$. Using the semi martingale representation and Riesz representation theorem for Hilbert space processes, we can show {\bf (1)} and {\bf (2)}, (\ref{fs4}), following the steps in \cite{ahmed-charalambous2012a} or   [Section~V,  \cite{charalambous-ahmedFIS_Parti2012}], for regular strategies. Moreover, (\ref{fs4si}) is  obtained by contradition.  Finally, {\bf (3)} is obtained precisely as in \cite{charalambous-ahmedFIS_Parti2012}. Finally, {\bf (4)} follows from the fact that the derivations of {\bf (1)-(3)} do not depend on the form of  the information structures generated via $z^i, i=1, \ldots, N$. 

\end{proof}

\ \

The important point to be made regarding Theorem~\ref{thmfs1} is that its derivation is based on applying, under the new (reference) probability space $\Big( {\Omega}, {\mathbb F}, {\mathbb F}_T, {\mathbb P}\Big)$, 
 any  method based on  strong formulation  (in our case \cite{ahmed-charalambous2012a,charalambous-ahmedFIS_Parti2012}), but with $u$ adapted to feedback information.

We also point out that the necessary conditions for a $u^o \in {\mathbb U}^{(N)}[0,T]$ to be  a PbP optimal  can be derived following the procedure described in   Theorem~\ref{thmfs1}, and that these necessary conditions are equivalent to the necessary conditions for team optimality,  as expected. 
These results are  stated as a Corollary.

 \begin{corollary} (Necessary Conditions for PbP Optimality under Reference Measure)\\
 \label{corollaryfs5.1}
   Suppose Assumptions~\ref{NC1} hold. Then we have  following.

\noi{\bf Necessary Conditions.}       For  an element $ u^o \in {\mathbb U}^{(N)}[0,T]$ with the corresponding solution $\Lambda^o \in B_{{\mathbb F}_T}^{\infty}([0,T], L^2(\Omega,{\mathbb R}))$ to be a PbP optimal strategy, it is necessary  that 
the statements of Theorem~\ref{thmfs1}, {\bf (1)-(4)}  hold, and statement {\bf (2)} corresponding to (\ref{fs4si}).

\end{corollary}

 \begin{proof} 
 The derivation is based on the procedure of Theorem~\ref{thmfs1}, but we only vary in the direction $u^i-u^{i,o}$, while the rest of the strategies are optimal, $u^{-i}=u^{-i,o}$.
\end{proof}

\ \

\noi{\bf Minimum Principle Under Original Probability Space $\Big( {\Omega}, {\mathbb F}, \{ {\mathbb F}_{0,t}: t \in [0, T]\}, {\mathbb P}^u\Big)$}\\

\noi Next,  we  express the optimality conditions with respect to the original probability space $\Big( {\Omega}, {\mathbb F}, \{ {\mathbb F}_{0,t}: t \in [0, T]\}, {\mathbb P}^u\Big)$. 

\noi Since the Hamiltonian under the reference probability measure (\ref{fs1}), appearing in Theorem~\ref{thmfs1} is multiplied by $\Lambda(\cdot)$, then we can write 
\bea
{\cal H}(t, x, \Lambda, \Psi, Q, u) =\Lambda \Big\{ Q \sigma^{-1}(t,x) f(t, x, u) + \ell(t,x, u)\Big\} \label{eqhe}
\eea
Define the Hamiltonian under the original probability measure ${\mathbb P}^u$ by  
\begin{align}
{\mathbb  H}: [0, T]  &\times  {\mathbb R}^{n}    \times {\cal L}( {\mathbb R}^{n}, {\mathbb R})  \times {\mathbb A}^{(N)} \longrightarrow {\mathbb R} \nonumber \\
{\mathbb H}(t,x,Q,u)  &\tri   \ell(t,x, u)
+  Q  \sigma^{-1}(t,x) f(t,x, u). \label{fs7}
\end{align}
Since $\Lambda(T) =\frac{ d{\mathbb P}^u}{d{\mathbb P}}|_{{\mathbb F}_{T}}$, then we can express the Hamiltonian system of equations (\ref{fs2}), (\ref{spu1}) under the original measure ${\mathbb P}^u$, by translating the martingale term, using the fact that
\begin{align}
W^u(t) \tri W(t)- \int_{0}^t \sigma^{-1}(s,x(s))f(s,x(s),u_s)ds, \hso \mbox{ is an}  \: \: \Big({\mathbb F}_{0,t}, {\mathbb P}^u\Big)-\mbox{martingale}. \label{mart-t}
\end{align} 
 Thus, under the original probability measure $\Big( {\Omega}, {\mathbb F}, \{ {\mathbb F}_{0,t}: t \in [0,T]\}  , {\mathbb P}^u\Big)$, by substituting (\ref{mart-t}) into (\ref{fs2}), (\ref{spu1}) the adjoint process $\{\Psi(t), Q(t): t \in [0, T]\}$ is a  solution of  the  backward and forward stochastic differential  equation
\begin{align}
d \Psi(t) = -\ell(t,x(t),u_t)dt  + Q(t) dW^u(t), \hso  \Psi(T)= \varphi(x(T)), \hso t \in [0,T),   \label{fs9}
\end{align}   
and the process $\{x(t): t \in [0,T]\}$ is a  solution of the following forward  equation.
\bea
dx(t)=f(t,x(t), u_t)dt+\sigma(t,x(t))dW^u(t), \hso x(0)=x_0. \label{fs10} 
\eea
Moreover, the conditional  variational  Hamiltonian  is given by  
\begin{align}
{\mathbb E}^{u^o} &\Big\{   {\mathbb H}(t,x^o(t),  Q^o(t), u_t^{-i,o}, u_t^{i})   |{\cal G}_{0, t}^{I^i} \Big\} \nonumber \\
&  \geq {\mathbb E}^{u^o} \Big\{   {\mathbb H}(t,x^o(t),  Q^o(t), u_t^{-i,o}, u_t^{i,o})   |{\cal G}_{0, t}^{I^i} \Big\}, \hso \forall u^i \in {\mathbb A}^i, \; a.e. t, \; {\mathbb P}^{u^o}|_{{\cal G}_{0,t}^{I^i}}- a.s., \forall  i \in {\mathbb Z}_N.  \label{fs11} 
\end{align}

Hence,  under the original probability space ${\mathbb P}^u$, we have the following necessary conditions for team optimality.

\begin{theorem}(Necessary Conditions for Team Optimality under Original Measure)\\
\label{thmfso1}
Suppose Assumptions~\ref{NC1} hold. Then we have the following.
 
\noi {\bf Necessary Conditions.}    For  an element $ u^o \in {\mathbb U}^{(N)}[0,T]$ with the corresponding solution $x^o \in B_{{\mathbb F}_T}^{\infty}([0,T], L^2(\Omega,{\mathbb R}^{n}))$ to be team optimal, it is necessary  that 
the following hold.

{\bf (1)}  There exists a semi martingale  $m^o \in {\cal SM}_0^2[0,T]$ ($1$-dimensional) with the intensity process $\{ ({\Psi}^o, Q^o)\} \in  L_{{\mathbb F}_T}^2([0,T],{\mathbb R})\times L_{{\mathbb F}_T}^2([0,T],{\cal L}({\mathbb R}^{n},{\mathbb R}))$.
 
{\bf (2) }  The variational inequalities are satisfied:
\begin{align}   
  \sum_{i=1}^N &{\mathbb  E}^{u^o} \Big\{ \int_0^T    {\mathbb H} (t,x^o(t), Q^{o}(t), u_t^{-i,o},u_t^i) dt \Big\}  \nonumber \\
     &\geq  \sum_{i=1}^N {\mathbb  E}^{u^{o}}  \Big\{ \int_0^T    {\mathbb H} (t,x^o(t), Q^o(t) u_t^{-i,o}, u_t^{i,o})) dt \Big\} ,   \hso \forall u \in {\mathbb U}^{(N)}[0,T]. \label{fs4o}
\end{align}
\begin{align}   
  {\mathbb  E}^{u^o} \Big\{ \int_0^T &   {\mathbb H} (t,x^o(t), Q^{o}(t), u_t^{-i,o},u_t^i) dt \Big\}  \nonumber \\
     \geq   {\mathbb  E}^{u^{o}} \Big\{ \int_0^T  &  {\mathbb H} (t,x^o(t), Q^o(t) u_t^{-i,o}, u_t^{i,o})) dt \Big\} ,   \hso \forall u^i \in {\mathbb U}^{I^i}[0,T]. \label{fs4one}
\end{align}

{\bf (3) } The process $\{({\Psi}^o, Q^o)\} \in  L_{{\mathbb F}_T}^2([0,T],{\mathbb R})\times L_{{\mathbb F}_T}^2([0,T],{\cal L}({\mathbb R}^{n}, {\mathbb R}))$ is a unique weak solution of the backward stochastic differential equations (\ref{fs9}) such that $u^o \in {\mathbb U}^{(N)}[0,T]$ satisfies  the  point wise almost sure inequalities with respect to the $\sigma$-algebras ${\cal G}_{0,t}^{I^i}   \subset {\mathbb F}_{0,t}$, $ t\in [0, T], i=1, 2, \ldots, N:$ 
\begin{align} 
& {\mathbb E}^{u^o} \Big\{   {\mathbb H}(t,x^o(t), Q^o(t), u_t^{-i,o},u_t^i)   |{\cal G}_{0, t}^{I^i} \Big\}  \nonumber \\
 & \geq  {\mathbb E}^{u^{o}} \Big\{   {\mathbb H}(t,x^o(t),  Q^o(t), u_t^{o})   |{\cal G}_{0, t}^{I^i} \Big\},  \hso \forall u^i \in {\mathbb A}^i,  a.e. t \in [0,T], {\mathbb P}^{u^o}|_{{\cal G}_{0,t}^{I^i}}- a.s., \forall i \in {\mathbb Z}_N.   \label{fs5o} 
\end{align} 

{\bf (4)}  For admissible strategies ${\mathbb U}^{(N), z}[0, T], {\mathbb U}^{(N), x}[0, T]$ the conditional expectation in (\ref{fs5o}) is taken with respect to the information structures ${\cal G}_{0,t}^{z^i}, {\cal G}^{x^i(t)}$, respectively.

\end{theorem}

\begin{proof}  The statements follow from  Theorem~\ref{thmfs1},  and the discussion prior to the Theorem.   
\end{proof}

\ \

For PbP optimality  we have the following  Corollary.

 \begin{corollary} (Necessary Conditions for PbP Optimality under Original Probability Meaure)\\
 \label{corollaryfs5.1o}
  Suppose Assumptions~\ref{NC1} hold. Then we have the following.

\noi {\bf Necessary Conditions.}       For  an element $ u^o \in {\mathbb U}^{(N)}[0,T]$ with the corresponding solution $x^o \in B_{{\mathbb F}_T}^{\infty}([0,T], L^2(\Omega,{\mathbb R}^{n}))$ to be a  PbP optimal strategy, it is necessary  that 
the statements of Theorem~\ref{thmfso1}, {\bf (1), (3), (4)}  hold and statement {\bf (2)} corresponding to (\ref{fs4one}).

\end{corollary}

 \begin{proof} Following from the change of probability measure and Corollary~\ref{corollaryfs5.1}.
\end{proof}

\ \

Therefore, we can  apply the necessary conditions for team optimality, either under  the reference  measure ${\mathbb P}$ or under original measure ${\mathbb P}^u$. 

In the next remark, we discuss the connection of Theorem~\ref{thmfso1} to the necessary conditions of optimality of stochastic control problems with centralized feedback information structures.

\begin{remark}
\label{remark5.2} 
From Theorem~\ref{thmfso1} one can deduce the  optimality conditions of classical stochastic control problems with centralized noiseless feedback information structures, that is, when for any $t \in [0,T],  u_t$ is measurable with respect  to the feedback information structure ${\cal F}_{0,t}^x \tri \sigma\{x(s): 0\leq s \leq t\}$, and to  partial noiseless feedback information structure ${\cal G}_{0,t}^x \subset {\cal F}_{0,t}^x$. Indeed,  the  necessary conditions for such a $u^o$ to be optimal with say, centralized partial noiseless feedback information structure, under the original probability measure ${\mathbb P}^u$ are
 
\begin{align}
 {\mathbb E}^{u^o} \Big\{  {\mathbb  H}&(t,x^o(t),Q^o(t),u )|{\cal G}_{0,t}^x\Big\}  \nonumber \\
&\geq   {\mathbb E}^{u^o}\Big\{  {\mathbb  H}(t,x^o(t), Q^o(t),u_t^o)|{\cal G}_{0,t}^x\Big\}, \hso \forall u \in {\mathbb A}^{(N)}, a.e. t \in [0,T], {\mathbb P}|_{{\cal G}_{0,t}^x}^{u^o}-a.s., \label{fs14} 
 \end{align} 
where $\{x^o(t), \Psi^o(t), Q^o(t): t \in [0,T]\}$ are the solutions of   (\ref{fs9}), (\ref{fs10}). Note that the Hamiltonian ${\mathbb H}(t,x,Q,u) \equiv Q \sigma^{-1}(t,x) f(t,x,u)+ \ell(t,x,u)$ is precisely the one derived in \cite{elliott1977}, using martingale methods.  \\
Clearly, (\ref{fs14}) generalizes the necessary conditions of optimality of classical stochastic control  problems with partial nonanticipative information structures \cite{ahmed-charalambous2012a,zhang-elliott-siu2012} and references therein, where  for any $t \in [0,T],  u_t$ is measurable with respect to any Brownian motion generated, sub$-\sigma-$algebras of ${\cal G}_{0,t}^W \subset {\cal F}_{0,t}^W \tri \sigma\{W(s): 0\leq s \leq t\}, t \in [0,T]$. 
\end{remark}

\subsection{Value Processes of Team Problems}
\label{scto}
In this section, we first show that the solution of the Backward stochastic differential equation is the value process of the stochastic dynamic team problem, lifted to a conditioning with respect to the centralized information structure. Then we  use the lifted value process to show that the necessary conditions (i.e.  (\ref{fs5o}))    for PbP  optimality are also sufficient.  

Define the sample  pay-off over the interval $[t, T]$ by 
\begin{align}
{\cal J}_{t,T}(u^1, \ldots, u^N) \tri   \int_{t}^T \ell(s, x(s), u_s)ds + \varphi(x(T)).  \label{sufo1}
\end{align}
and its conditional expectation with respect to ${\cal G}_{0,t}^{I^i}, i=1, \ldots, N$,  by 
\begin{align}
J_{t,T}^i(u) \tri {\mathbb E}^u \Big\{ {\cal J}_{t,T}(u^1, \ldots, u^N) | {\cal G}_{0,t}^{I^i}\Big\}, \; u \in {\mathbb U}^{(N)}[t, T],  \label{sufo2}
\end{align}
where $ {\mathbb U}^{(N)}[t, T]$ is the restriction of the strategies ${\mathbb U}^{(N)}[0, T]$, to the interval $[t,T]$.\\
  PbP optimality seeks admissible strategies $u^i \in  {\mathbb U}^{I^i}[t,T], i=1, \ldots, N$ to minimize  the pay-off, in the sense, 
\begin{align}
{\mathbb E}^{u^{-i,o}, u^{i,o}} \Big\{ {\cal J}_{t,T}(u^{-i,o}, u^{i,o}) | {\cal G}_{0,t}^{I^i}\Big\}  \leq  {\mathbb E}^{u^{-i,o}, u^{i}} \Big\{ {\cal J}_{t,T}(u^{-i,o}, u^{i}) | {\cal G}_{0,t}^{I^i}\Big\}, \; \forall u^i \in  {\mathbb U}^{I^i}[t,T], i=1, \ldots, N.  \nonumber 
\end{align}
This means that when team members employ strategies, $u^{-i, o} \in \times_{j=1, \j \neq i}^N {\mathbb U}^{I^j}[t,T]$, team member $u^i$ minimizes the reward $ {\mathbb E}^{u^{-i,o}, u^{i}} \Big\{ {\cal J}_{t,T}(u^{-i,o}, u^{i}) | {\cal G}_{0,t}^{I^i}\Big\}$ over all strategies ${\mathbb U}^{I^i}[0,T]$. 
The set of all such strategies $(u^{1,0}, \ldots, u^{N,0}) \in \times_{i=1}^N  {\mathbb U}^{I^i}[t, T]$ is called PbP optimal. \\ 
We denote the value processes of the team game for each team member by
\begin{align}
V^i(t) \tri {\mathbb E}^{u^{-i,o}, u^{i,o}} \Big\{ {\cal J}_{t,T}(u^{-i,o}, u^{i,o}) | {\cal G}_{0,t}^{I^i}\Big\}  , \: i=1, \ldots, N. \label{sufo4}
\end{align}

Consider the solution of the  backward stochastic differential equation (\ref{fs2})  
\begin{align}
\Psi^{u}(t) =\Psi^u(T) + \int_{t}^T {\mathbb H}(s, x(s), Q^u(s), u_s)ds -\int_{t}^T Q^u(s)dW(s), \hst t \in [0,T).  \label{sufo11}
\end{align}
For $u = u^o$ this is the lifted value process of the team pay-off with respect to the information ${\mathbb F}_{0,t}, t \in [0, T]$. From (\ref{sufo11}) we have 
\begin{align}
\Psi^{u}(t) =\Psi^u(T) + \int_{t}^T  \ell((s, x(s), u_s)ds -\int_{t}^T Q^u(s)dW^u(s), \hst t \in [0,T),  \label{sufo22}
\end{align}
and by taking conditional expectation ${\mathbb E}^u\Big\{ \cdot | {\mathbb F}_{0,t}\Big\}$ of both sides of (\ref{sufo22}), and using $\Psi^u(T)=\varphi(x(T))$, we obtain
\begin{align}
\Psi^{u}(t) =  {\mathbb E}^u  \Big\{ \int_{t}^T \ell(s, x(s), u_s)ds + \varphi(x(T)) |  {\mathbb F}_{0,t}\Big\}.  \label{sufo33}
\end{align}
Hence, 
\begin{align}
J_{t,T}^i(u) = {\mathbb E}^u \Big\{ \Psi^{u}(t)| {\cal G}_{0,t}^{I^i}\Big\}, \; u \in {\mathbb U}^{(N)}[0,T], \forall i \in {\mathbb Z}_N.  \label{sufo333}
\end{align}
Now, we state the main theorem.

\begin{theorem}(Sufficient Conditions for PbP Optimality)
\label{sufpbp}
Let $$(\Psi^{u}, Q^u) \in L^2([0,T], L^2(\Omega, {\mathbb R})) \times L^2([0,T], L^2(\Omega, {\cal L}({\mathbb R}^n, {\mathbb R})))$$ be  a solution of the backward stochastic differential equation (\ref{sufo11}).\\
If $u^{i,o} \in {\mathbb U}^{I^i}[t, T]$ satisfy the conditional variational inequalities (\ref{fs5o}), then $(u^{1,o},\ldots, u^{N.o}) \in \times_{i=1}^N {\mathbb U}^{I^i}[t, T]$ is a PbP optimal. \\
Moreover,   $a.e. t \in [0,T], {\mathbb P}^{u^{-i,o}, u^{i,o}} |_{{\cal G}_{0,t}^{I^i}}-$a.s. we have 
\begin{align}
V^i(t) = {\mathbb  E}^{u^{-i,o}, u^{i,o}} \Big\{ {\mathbb E}^{u^{-i,o}, u^{i,o}}  \Big\{ \Psi^{u^{-i,o}, u^{i,o}}(t) | {\mathbb F}_{0,t}\Big\} | {\cal G}_{0,t}^{I^i} \Big\} 
= {\mathbb  E}^{u^{-i,o}, u^{i,o}} \Big\{ \Psi^{u^{-i,o}, u^{i,o}}(t)  | {\cal G}_{0,t}^{I^i} \Big\}, \: i=1, \ldots, N. \label{sufo5}
\end{align} 
\end{theorem}

\begin{proof} Since  (\ref{fs5o}) holds, by taking expectation on both sides we deduce\\ $
{\mathbb E}^{u^o} \Big\{   {\mathbb H}(t,x^o(t), Q^o(t), u_t^{-i,o},u_t^i)   \Big\} 
  \geq  {\mathbb E}^{u^{o}} \Big\{   {\mathbb  H}(t,x^o(t),  Q^o(t), u_t^{o})   \Big\}, 
   \forall u^i \in {\mathbb A}^i, \forall i \in {\mathbb Z}_N.$ This implies that for almost every $(t, x) \in [0,T] \times C([0,T], {\mathbb R}^n)$,   
\begin{align}
  {\mathbb H}(t,x^o(t), Q^o(t), u_t^{-i,o},u_t^i)    \geq   {\mathbb  H}(t,x^o(t),  Q^o(t), u_t^{o}), \hso  
   \forall u^i \in {\mathbb A}^i, \forall i \in {\mathbb Z}_N. \label{sufo6}
\end{align}
Therefore, an application of the comparison theorem of stochastic differential equations \cite{karatzas-shreve1991} to the second right hand side term of (\ref{sufo11}) yields $\Psi^{u^{i,o}, u^i}(\cdot) \geq \Psi^{u^{-i,o}, u^{i,o}}(\cdot),$ a.e.  on $[0, T]  \times C([0,T], {\mathbb R}^n)$. From (\ref{sufo33}), we have 

\begin{align}
\Psi^{u^{-i,o}, u^{i,o}}(t) =&  {\mathbb E}^{u^{-i,o}, u^{i,o}}  \Big\{ \int_{t}^T \ell(s, x(s), u_s^{-i,o}, u_s^{i,o})ds + \varphi(x(T)) |  {\mathbb F}_{0,t}\Big\}  \label{sufo7} \\
\leq & \Psi^{u^{-i,o}, u^{i}}(t) \nonumber \\
=& {\mathbb E}^{u^{-i,o}, u^{i}}  \Big\{ \int_{t}^T \ell(s, x(s), u_s^{-i,o}, u_s^{i})ds  + \varphi(x(T)) |  {\mathbb F}_{0,t}\Big\},   \hst \forall u^i \in {\mathbb U}^{I^i}[0,T].    \label{sufo8}  
\end{align}
By taking conditional expectation of both sides of last inequality with respect to ${\cal G}_{0,t}^{I^i}$, we obtain  
\begin{align}
{\mathbb E}^{u^{-i,o}, u^{i,o}} & \Big\{ \int_{t}^T \ell(s, x(s), u_s^{-i,o}, u_s^{i,o})ds  
+ \varphi(x(T)) |  {\cal G}_{0,t}^{I^i}\Big\}
\leq  {\mathbb E}^{u^{-i,o}, u^{i}}  \Big\{ \int_{t}^T \ell(s, x(s), u_s^{-i,o}, u_s^{i})ds  \nonumber \\
&+ \varphi(x(T)) |  {\cal G}_{0,t}^{I^i}\Big\},   \hst \forall u^i \in {\mathbb U}^{I^i}[0,T].    \label{sufo9}  
\end{align}
Since this hods for all $i=1, \ldots, N$ we deduce PbP optimality. \\
Finally, by taking conditional expectation of both sides of (\ref{sufo7})  with respect to ${\cal G}_{0,t}^{I^i}$ we deduce (\ref{sufo5}). This completes the derivation.

\end{proof}

Finally, we note that if we consider the extended state $(x, \Lambda)$ and corresponding Hamiltonian system of equations, under certain global convexity conditions (see \cite{charalambous-ahmedPISG2013a}), we can show that PbP optimality implies team optimality.

\section{Conclusions and Future Work}
\label{cf}
This paper generalizes static team theory to  stochastic differential decision system with decentralized noiseless feedback information structures.    We have applied Girsanov's theorem to transformed the initial dynamic team problem  to an equivalent team problem, under a reference probability space, with state process independent of any of the team decisions. Then, we described the connection to static team theory discussed by Witsenhausen in \cite{witsenhausen1988}, and we proceeded further to   derive team and PbP  optimality  conditions, using  the stochastic Pontryagin's maximum principle.
 We also discussed the connection between the backward stochastic differential equation and the value process of the team problem.\\
In future work we will apply the optimality conditions to problems from the communication and control areas, as in \cite{charalambous-ahmedFIS_Partii2012}, but instead of decentralized  nonanticipative strategies, we will use decentralized feedback strategies.

\section{Appendix}
{\bf Proof of Theorem~\ref{lemma-g}.} 

 First, we show that ${\mathbb E} \Big\{ \Lambda^u(t) |x(t)|_{{\mathbb R}^n}^2\Big\} < K$. By applying the It\^o differential rule

\begin{align}
d |x(t)|_{{\mathbb R}^n}^2 =& 2\la x(t), \sigma(t,x(t))dW(t) \ra dt 
+ tr \Big( a(t,x(t))\Big)dt, \hst a(t,x) \tri \sigma(t,x) \sigma^*(t,x)  \label{gs3} \\
d \Big(\Lambda_t  |x(t)|_{{\mathbb R}^n}^2\Big) =& \Lambda^u(t)  |x(t)|_{{\mathbb R}^n}^2 f^{*}(t,x(t),u_t) \Big(a(t,x(t))\Big)^{-1} dx(t)  + 2 \Lambda^u(t) \la x(t), \sigma(t,x(t))dW(t) \ra \nonumber \\
&+    2 \Lambda^u(t) \la x(t), f(t,x(t),u_t) \ra + \Lambda^u(t) tr \Big( a(t,x(t)) \Big)dt. \label{gs4} 
\end{align}
Then by applying the It\^o differential rule once more we have
\begin{align}
&d \frac{ \Lambda^u(t)  |x(t)|_{{\mathbb R}^n}^2}{ 1+ \epsilon \Lambda^u(t)  |x(t)|_{{\mathbb R}^n}^2} = \frac{1}{\Big( 1+ \epsilon \Lambda^u(t)  |x(t)|_{{\mathbb R}^n}^2\Big)^2} \Big\{  \Lambda^u(t)  |x(t)|_{{\mathbb R}^n}^2 f^{*}(t,x(t))\Big(a(t,x(t))\Big)^{-1} dx(t) \nonumber \\
& + 2 \Lambda^u(t) \la x(t), \sigma(t,x(t))dW(t) \ra 
+ 2 \Lambda^u(t)
\la x(t), f(t,x(t),u_t) \ra dt+\Lambda^u(t) tr \Big( a(t,x(t)) \Big)dt \Big\} \nonumber \\
&- \frac{\epsilon (\Lambda^u(t))^2}{\Big( 1+ \epsilon \Lambda^u(t)  |x(t)|_{{\mathbb R}^n}^2\Big)^3}\Big\{ \la \Big(  ( a(t,x(t))\Big)^{-1} f(t,x(t),u_t) |x(t)|_{{\mathbb R}^n}^2 \nonumber \\
&+2x(t), a(t,x(t)) \Big(  \Big(   a(t,x(t))\Big)^{-1} f(t,x(t),u_t) |x(t)|_{{\mathbb R}^n}^2+2x(t) \Big)\ra  \Big\}dt \nonumber 
\end{align}

Integrating over $[0,T]$ and taking the expectation with respect to ${\mathbb P}$, and using the fact that ${\mathbb E} \Big( \Lambda^u(t)\Big) \leq 1, \forall t \in [0,T]$, yields
\begin{align}
\frac{d}{dt} {\mathbb E} \Big\{  \frac{ \Lambda^u(t)  |x(t)|_{{\mathbb R}^n}^2}{ 1+ \epsilon \Lambda^u(t)  |x(t)|_{{\mathbb R}^n}^2} \Big\} \leq  {\mathbb E} \Big\{ \frac{\Lambda^u(t) \Big[ 2\la x(t), f(t,x(t),u_t) \ra + tr \Big( a(t,x(t)) \Big) \Big]     }{ 1+ \epsilon \Lambda^u(t)  |x(t)|_{{\mathbb R}^n}^2} \Big\}. \label{gs6} 
 \end{align}
By Assumptions~\ref{A1-A4}, {\bf (A1), (A2)}, there exists $K>0$ such that 
\begin{align}
\frac{d}{dt} {\mathbb E} \Big\{  \frac{ \Lambda^u(t)  |x(t)|_{{\mathbb R}^n}^2}{ 1+ \epsilon \Lambda^u(t)  |x(t)|_{{\mathbb R}^n}^2} \Big\} \leq K \Big(  {\mathbb E}  \Big\{ \frac{\Lambda^u(t) \Big[ |x(t)|_{{\mathbb R}^n}^2 + |u_t|_{{\mathbb R}^d}^2\Big]   }{ 1+ \epsilon \Lambda^u(t)  |x(t)|_{{\mathbb R}^n}^2} \Big\}  +1 \Big) \nonumber 
 \end{align}
Since for any $u \in {\mathbb U}^{(N)}[0, T]$,  we have  $ {\mathbb E} \int_{0}^T \Lambda^u(t)  |u_t|_{{\mathbb R}^d}^2 dt$ is finite, then it follows from Gronwall inequality that 
  \begin{align}
{\mathbb E}\Big\{  \frac{ \Lambda^u(t)  |x(t)|_{{\mathbb R}^n}^2}{ 1+ \epsilon \Lambda^u(t)  |x(t)|_{{\mathbb R}^n}^2} \Big\} \leq C, \hst \forall t \in [0,T]. \label{gs7}
 \end{align}
 By Fatou's lemma we obtain ${\mathbb E} \Big\{\Lambda^u(t)  |x(t)|_{{\mathbb R}^n}^2\Big\}< C,  \forall t  \in [0,T]$. 
 Consider
 \begin{align}
d \frac{\Lambda^u(t)}{ 1 + \epsilon \Lambda^u(t)} = & \frac{\Lambda^u(t) f^*(t,x(t))\Big(a(s,x(s)) \Big)^{-1} dx(t)}{\Big( 1 + \epsilon \Lambda^u(t)\Big)^2}  \nonumber \\
&-   \frac{\epsilon (\Lambda^u(t))^2 f^*(t,x(t),u_t)\Big(a(t,x(t))\Big)^{-1}f(t,x(t),u_t)}{\Big( 1 + \epsilon \Lambda^u(t)\Big)^3}, \nonumber 
\end{align}
then 
\begin{align}
 {\mathbb E} \frac{\Lambda^u(t)}{ 1 + \epsilon \Lambda^u(t)} = 1-  {\mathbb E} \int_{0}^t \frac{\epsilon (\Lambda^u(s))^2 f^*(s,x(s),u_s)\Big( a(t,x(t)) \Big)^{-1}
 f(s,x(s),u_s)}{\Big( 1 + \epsilon \Lambda^u(s)\Big)^3}ds. \label{geq1} 
\end{align}

 Since 
\bes 
 \frac{\epsilon (\Lambda^u(t))^2 f^*(t,x(t),u_t)\Big( a(t,x(t))\Big)^{-1}f(t,x(t),u_t)}{\Big( 1 + \epsilon \Lambda^u(t)\Big)^3} \longrightarrow 0, \hso a.e. \: t \in [0,T],\: {\mathbb P}-a.s. \hso \mbox{as} \hso \epsilon \longrightarrow 0,
\ees 
and by {\bf (A7)} there exists a constant $C>0$ such that it is bounded  by $C \Lambda^u(t) \Big( 1+|x(t)|_{{\mathbb R}^n}^2+ |u_t|_{{\mathbb R}^d}^2  \Big)$,  then by the Lebesgue's  dominated convergence theorem we have
 \bea 
{\mathbb E} \int_{0}^t \frac{\epsilon (\Lambda^u(s))^2 f^*(s,x(s),u_s)\Big( a(s,x(s))\Big)^{-1}f(s,x(s),u_s)}{\Big( 1 + \epsilon \Lambda^u(s)\Big)^3} \longrightarrow 0 \hso \mbox{as} \hso \epsilon \longrightarrow 0. \label{geq2}
\eea 
 Since ${\mathbb E} \Big(\Lambda^u(t) \Big)\leq 1, \forall t \in [0,T]$ then by using (\ref{geq1}) into (\ref{geq2}), we obtain  ${\mathbb E} \frac{\Lambda^u(t)}{ 1 + \epsilon \Lambda^u(t)} \longrightarrow {\mathbb E} \Lambda^u(t)$, as $\epsilon \longrightarrow 0$. Hence, we must have  ${\mathbb E} \Lambda^u(t)=1, \forall t \in [0,T]$. Consequently, we have equivalence of the two dynamic team problems.

\bibliographystyle{IEEEtran}
\bibliography{bibdata}

\end{document}